\documentclass[reqno]{amsart}
\usepackage{amssymb,amsthm,amsfonts,amstext,amsmath}
\usepackage{mathtools}
\usepackage{enumerate}
\usepackage{fontawesome5}
\usepackage{eso-pic}
\usepackage{charter}
\usepackage{listings}
\usepackage{color}
\usepackage{textcomp}
\usepackage[normalem]{ulem}

\definecolor{dkgreen}{rgb}{0,0.6,0}
\definecolor{gray}{rgb}{0.5,0.5,0.5}
\definecolor{mauve}{rgb}{0.58,0,0.82}
\lstdefinelanguage{none}{
	identifierstyle=
}

\usepackage{comment}
\usepackage{graphicx}
\usepackage{enumerate}
\numberwithin{equation}{section}
\usepackage{float}
\usepackage[colorlinks=true,linkcolor=blue,citecolor=magenta]{hyperref}

\makeatletter
\def\namedlabel#1#2{\begingroup
	#2%
	\def\@currentlabel{#2}%
	\phantomsection\label{#1}\endgroup
}
\makeatother

\usepackage{mathrsfs}
\usepackage{pythonhighlight}
\usepackage{tikz}
\usepackage{dsfont}
\newlength{\starsize}
\newlength{\starspread}
\tikzset{starsize/.code={\setlength{\starsize}{#1}},
	starspread/.code={\setlength{\starspread}{#1}}}
\tikzset{starsize=1mm,
	starspread=3mm}
\pgfdeclarepatternformonly[\starspread,\starsize]% variables
{custom fivepointed stars}% name
{\pgfpointorigin}% lower left corner
{\pgfqpoint{\starspread}{\starspread}}% upper right corner
{\pgfqpoint{\starspread}{\starspread}}% tilesize
{% shape description
	\pgftransformshift{\pgfqpoint{\starsize}{\starsize}}
	\pgfpathmoveto{\pgfqpointpolar{18}{\starsize}}
	\pgfpathlineto{\pgfqpointpolar{162}{\starsize}}
	\pgfpathlineto{\pgfqpointpolar{306}{\starsize}}
	\pgfpathlineto{\pgfqpointpolar{90}{\starsize}}
	\pgfpathlineto{\pgfqpointpolar{234}{\starsize}}
	\pgfpathclose%
	\pgfusepath{fill}
}

\usepackage[utf8]{inputenc}
\usepackage[T1]{fontenc}
\usetikzlibrary{backgrounds}
\usetikzlibrary{patterns,fadings}
\usetikzlibrary{arrows,decorations.pathmorphing}
\usetikzlibrary{calc}
\definecolor{light-gray}{gray}{0.95}

\usetikzlibrary{patterns}

\newlength{\hatchspread}
\newlength{\hatchthickness}
\newlength{\hatchshift}
\newcommand{\hatchcolor}{}
\tikzset{hatchspread/.code={\setlength{\hatchspread}{#1}},
	hatchthickness/.code={\setlength{\hatchthickness}{#1}},
	hatchshift/.code={\setlength{\hatchshift}{#1}},% must be >= 0
	hatchcolor/.code={\renewcommand{\hatchcolor}{#1}}}
\tikzset{hatchspread=3pt,
	hatchthickness=0.4pt,
	hatchshift=0pt,% must be >= 0
	hatchcolor=black}
\pgfdeclarepatternformonly[\hatchspread,\hatchthickness,\hatchshift,\hatchcolor]% variables
{custom north west lines}% name
{\pgfqpoint{\dimexpr-2\hatchthickness}{\dimexpr-2\hatchthickness}}% lower left corner
{\pgfqpoint{\dimexpr\hatchspread+2\hatchthickness}{\dimexpr\hatchspread+2\hatchthickness}}% upper right corner
{\pgfqpoint{\dimexpr\hatchspread}{\dimexpr\hatchspread}}% tile size
{% shape description
	\pgfsetlinewidth{\hatchthickness}
	\pgfpathmoveto{\pgfqpoint{0pt}{\dimexpr\hatchspread+\hatchshift}}
	\pgfpathlineto{\pgfqpoint{\dimexpr\hatchspread+0.15pt+\hatchshift}{-0.15pt}}
	\ifdim \hatchshift > 0pt
	\pgfpathmoveto{\pgfqpoint{0pt}{\hatchshift}}
	\pgfpathlineto{\pgfqpoint{\dimexpr0.15pt+\hatchshift}{-0.15pt}}
	\fi
	\pgfsetstrokecolor{\hatchcolor}
	\pgfusepath{stroke}
}

\pgfdeclarepatternformonly[\hatchspread,\hatchthickness,\hatchshift,\hatchcolor]% variables
{custom north east lines}% name
{\pgfqpoint{\dimexpr-2\hatchthickness}{\dimexpr-2\hatchthickness}}% lower left corner
{\pgfqpoint{\dimexpr\hatchspread+2\hatchthickness}{\dimexpr\hatchspread+2\hatchthickness}}% upper right corner
{\pgfqpoint{\dimexpr\hatchspread}{\dimexpr\hatchspread}}% tile size
{% shape description
	\pgfsetlinewidth{\hatchthickness}
	\pgfpathmoveto{\pgfqpoint{\dimexpr\hatchshift-0.15pt}{-0.15pt}}
	\pgfpathlineto{\pgfqpoint{\dimexpr\hatchspread+0.15pt}{\dimexpr\hatchspread-\hatchshift+0.15pt}}
	\ifdim \hatchshift > 0pt
	\pgfpathmoveto{\pgfqpoint{-0.15pt}{\dimexpr\hatchspread-\hatchshift-0.15pt}}
	\pgfpathlineto{\pgfqpoint{\dimexpr\hatchshift+0.15pt}{\dimexpr\hatchspread+0.15pt}}
	\fi
	\pgfsetstrokecolor{\hatchcolor}
	\pgfusepath{stroke}
}

\usepackage{xargs}                      % Use more than one optional parameter in a new commands
\usepackage[colorinlistoftodos,prependcaption,textsize=tiny]{todonotes}
\newcommandx{\unsure}[2][1=]{\todo[linecolor=red,backgroundcolor=red!25,bordercolor=red,#1]{#2}}
\newcommandx{\change}[2][1=]{\todo[linecolor=blue,backgroundcolor=blue!25,bordercolor=blue,#1]{#2}}
\newcommandx{\info}[2][1=]{\todo[linecolor=OliveGreen,backgroundcolor=OliveGreen!25,bordercolor=OliveGreen,#1]{#2}}
\newcommandx{\improvement}[2][1=]{\todo[linecolor=Plum,backgroundcolor=Plum!25,bordercolor=Plum,#1]{#2}}
\newcommandx{\thiswillnotshow}[2][1=]{\todo[disable,#1]{#2}}

\usepackage{float}
\usepackage[colorlinks=true,linkcolor=blue,citecolor=magenta]{hyperref}
\newtheorem{teorema}{Theorem}[section]
\newtheorem{proposicao}[teorema]{Proposition}
\newtheorem{lema}[teorema]{Lemma}
\newtheorem{definicao}[teorema]{Definition}

\theoremstyle{definition}
\newtheorem{observacao}{Remark}

\newcommand{\msf}[1]{{\mathsf #1}}
\newcommand{\mc}[1]{{\mathcal #1}}
\newcommand{\mf}[1]{{\mathfrak #1}}

\newcommand{\bb}[1]{{\mathbb #1}}

\newcommand{\eps}{\varepsilon}

\newcommand{\N}{\mathbb N}
\newcommand{\R}{\mathbb R}

\def\centerarc[#1](#2)(#3:#4:#5){\draw[#1] ($(#2)+({#5*cos(#3)},{#5*sin(#3)})$) arc (#3:#4:#5);}
\newcommand{\pfrac}[2]{\genfrac{}{}{}{1}{#1}{#2}}

\newcommand{\function}[3]{{#1: #2 \to #3}}

\def\II{\mathrm{I\kern-0.1emI}}

\let\oldtocsection=\tocsection
\let\oldtocsubsection=\tocsubsection
\let\oldtocsubsubsection=\tocsubsubsection
\renewcommand{\tocsection}[2]{\hspace{0em}\oldtocsection{#1}{#2}}
\renewcommand{\tocsubsection}[2]{\hspace{1em}\oldtocsubsection{#1}{#2}}
\renewcommand{\tocsubsubsection}[2]{\hspace{2em}\oldtocsubsubsection{#1}{#2}}
\DeclareRobustCommand{\SkipTocEntry}[5]{}
\keywords{infinitesimal generator, general Brownian motion, functional central limit theorem}

\begin{document}

\title[The most general BM on the line and on two half-lines]{The most general Brownian motion on\\ the line and on two closed half-lines}

\author[D. Erhard]{Dirk Erhard}
\address{UFBA\\
 Instituto de Matem\'atica, Campus de Ondina, Av. Milton Santos, S/N. CEP 40170-110\\
Salvador, Brazil}
\curraddr{}
\email{dirk.erhard@ufba.br}
\thanks{}

\author[T. Franco]{Tertuliano Franco}
\address{UFBA\\
 Instituto de Matem\'atica, Campus de Ondina, Av. Milton Santos, S/N. CEP 40170-110\\
Salvador, Brazil}
\curraddr{}
\email{tertu@ufba.br}
\thanks{}

\author[W. Muricy]{Wanessa Muricy}
\address{IMPA\\
	Estrada Dona Castorina, 110
	Jardim Botânico
	CEP 22460-320
	Rio de Janeiro, Brazil
	}
\curraddr{}
\email{wanessa.muricy@gmail.com}
\thanks{}

\subjclass[2010]{60J65, 60G53}

\begin{abstract}
	
	In the 1950s, W. Feller characterized the most general Brownian motion on the closed half-line $[0,\infty)$. He showed that any such process is a mixture of reflected, sticky, and killed Brownian motions. By “most general Brownian motion” we mean a càdlàg strong Markov process whose excursions away from zero coincide with those of standard Brownian motion, and which may be sent to the cemetery state upon hitting zero.	
	In this work, we fully characterize the most general Brownian motion on the whole real line and on the union of two closed half-lines. Our results are twofold. First, we show that the most general Brownian motion on $\bb R$ is the process known in the literature as the \textit{skew sticky Brownian motion killed at zero} (see Borodin and Salminen's book \cite[page 127, Section 13, Appendix 1]{Borodin}). Second, we prove that the most general Brownian motion on two closed half-lines is a process, which we call the \textit{skew sticky killed at zero snapping out Brownian motion}. This process extends the snapping out Brownian motion introduced by A.\ Lejay~\cite{Lejay} in 2016.
\end{abstract}

\maketitle

\tableofcontents

\allowdisplaybreaks

\section{Introduction}\label{s1}

The most general Brownian motion on the positive half-line was studied by Feller, and a comprehensive overview on it can be found in Knight's book, see \cite{Knight}. It is defined as a class of strong Markov processes on the positive half line such that they have the same law as Brownian motion until the first time that they hit zero. This can be shown to be a mixture of the reflected Brownian motion,  absorbed Brownian motion, and killed Brownian motion. 
 Denote by $\bb R_{+,\Delta} = [0,\infty)\cup \{\Delta\}$ the one-point compactification of $[0,\infty)$, where $\Delta$ is called the cemetery. Moreover, $C(\bb R_{+,\Delta})$ is the space of continuous functions defined on $\bb R_{+,\Delta}$, where the usual convention $f(\Delta)=0$ is used.
\begin{teorema}[Feller, see \cite{Knight}, Theorem 6.2, p.\ 157]\label{dgBM}
	Any general Brownian motion $ W$ on $\bb R_{+,\Delta}$ has generator $\msf L=\frac{1}{2}\frac{d^2}{dx^2}$ with corresponding domain
	\begin{equation*}%\label{domainBM}
		\mf{D}\big(\msf{L}\big)=\Big\{f\,\in\,C(\bb R_{+,\Delta})\,:f''\in C(\bb R_{+,\Delta})\, \text{ and } c_1f(0)-c_2f'(0)+\frac{c_3}{2}f''(0)=0\Big\}
	\end{equation*}
	for some $c_i\geq 0$ such that $c_1 + c_2 + c_3 = 1$ and $c_1 \neq 1$.
\end{teorema}

\begin{figure}[!htb]
	\centering
	\begin{tikzpicture}[scale=1]
		
		\coordinate (A) at (-1,0);
		\coordinate (B) at (1,3.46);
		\coordinate (C) at (3,0);
		\coordinate (E) at (1,1);
		
		\draw (A) -- node[midway, above, rotate=60, align = center] {\small{elastic BM}\\ $c_3=0$} (B);
		\draw (B) -- node[midway, above, rotate=-60, align = center] {\small{sticky BM}\\ $c_1=0$} (C);
		\draw (C) -- node[midway, below, align = center] {\small{exponential holding BM}\\ $c_2=0$} (A);
		\draw[very thick] (A) -- (B) -- (C) -- cycle;
		\fill[gray!20] (A) -- (B) -- (C) -- cycle;
		\draw (E)  node[align = center]{\small{mixed BM} \\$c_1, c_2, c_3>0$};
		\draw (A) node[left, align = center]{\small{killed BM}\\ $c_1=1$};
		\draw (B)+(0,0.5) node[align=center]{\small{reflected BM}\\$c_2=1$};
		\draw (C) node[right, align = center]{\small{absorbed BM}\\ $c_3=1$};
		\filldraw[fill=white] (A) circle (3pt);
		\filldraw[fill=blue] (B) circle (3pt);
		\filldraw[fill=red] (C) circle (3pt);
	\end{tikzpicture}
	\caption{Description of the general Brownian motion on the half-line according to the chosen values on the simplex $c_1+c_2+c_3=1$.}
	\label{FigSimplex}
\end{figure}
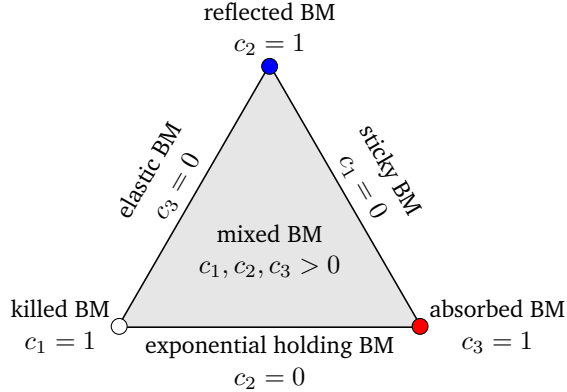

The simplex described in Theorem~\ref{dgBM} corresponds to the situation illustrated in the Figure~\ref{FigSimplex}.
The case $c_2 = 1$ corresponds to the \emph{reflected Brownian motion} (which has the distribution of the modulus of a standard BM), while $c_3 = 1$ yields the absorbed Brownian motion, which has the distribution of a standard BM stopped at zero. The extremal case $c_1=1$ is not part of Feller's Theorem~\ref{dgBM}, which imposes $c_1\neq 1$.   Indeed, for $c_1=1$ the set $\mf{D}\big(\msf{L}\big)$ is not  dense in the set of continuous functions on $\{\Delta\}\cup [0,\infty)$ decaying at infinity. Hence, it cannot be the domain of a generator on the closed half-line.  See  \cite[Chapter 2]{Chung} for more on the killed BM.

In this work, we extend Feller’s theorem to the setting of the full real line $\bb R\cup\{\Delta\}$
and to the setting of two closed half-lines $(-\infty,0-]\cup [0+,\infty)\cup\{\Delta\}$. We introduce a \textit{precise notion} of the most general Brownian motion in these domains, \textit{fully characterize} its generator, and establish its existence. 

On the real line, we prove that the most general Brownian motion is the \textit{Skew Sticky Killed at zero BM}, whose generator can be found in the book by Borodin and Salminen \cite[page 127, Section 13, Appendix 1]{Borodin}. So we conclude that no new process can be extracted from this setting, and the already known Skew Sticky Killed at zero BM is indeed the most general BM on the real line.

On two closed half-lines, we obtain a class of processes that generalizes the Snapping Out Brownian Motion (SNOB) constructed by A.\ Lejay in 2016, see \cite{Lejay}. The state space of the SNOB is $(-\infty,0-]\cup[0+, \infty)$ and has a behavior that can be described as follows. In each half-line, the process behaves as a reflected BM. However, when the local time of the process at $0+$ (or $0-$) reaches a value given by an independent exponential random variable, the process exchange the half-line and starts again. The class of process we achieve is a mixture of skew Brownian motion, sticky Brownian motion, killed at zero Brownian motion and the SNOB. 

As a consequence of our complete characterization of the generator of the most general Brownian motion on two half-lines, we establish a functional central limit theorem for a class of random walks on $\{\ldots, \frac{-2}{n}, \frac{-1}{n}, \frac{0-}{n}\} \cup \{\Delta\} \cup \{\frac{0+}{n}, \frac{1}{n}, \frac{2}{n}, \ldots\}$ whose behaviour at $0+$ and $0-$ involves a mixture of all the behaviours described above. Its limit is the \textit{Skew Sticky Killed at Zero Snapping Out Brownian Motion.} Computational simulations are also provided to illustrate this behavior.

\subsection{Related literature}
\subsubsection*{Brownian motion with special behaviour at zero}
Brownian motion with modified behaviour at the origin has been a central topic of study since Feller’s seminal works. A classical example appears in Portenko’s analysis \cite{Portenko1976} (see also \cite[Theorem~3.4, p.~146]{Portenko1990}) of the stochastic differential equation
\[
dY_{\varepsilon}(t) \;=\; b_{\varepsilon}(Y_{\varepsilon}(t))\,dt + dw(t)\,, \qquad t \ge 0,\; Y_\varepsilon(0)=y\,,
\]
where $w$ denotes Brownian motion, $b_\varepsilon(x)=L_\varepsilon \varepsilon^{-1} b(\varepsilon^{-1}x)$, and $b$ is an integrable function supported on $[-1,1]$.
When $L_\varepsilon$ is constant, Portenko proved that $Y_\varepsilon$ converges as $\varepsilon \to 0$ to skew Brownian motion, introduced in \cite{ItoMcKean1974}.  
If instead $L_\varepsilon \to \infty$ at a suitable rate, Mandrekar and Pilipenko \cite{MandrekarPilipenko2016} showed that the limit is a \emph{Brownian motion with a hard membrane}, also known as SNOB, a term coined by Lejay in~\cite{Lejay}.

The SNOB has since attracted considerable attention. Bobrowski \cite{Bobrowski2015,Bobrowski2016} proved that SNOB with large permeability converges to skew Brownian motion, and in \cite{BobrowskiRatajczyk2025} quantitative convergence rates are obtained.
In~\cite{EFS2020}, Erhard, Franco and Silva studied a nearest-neighbour random walk on $\mathbb{Z}$ with the jump rates illustrated in Figure~\ref{fig1a}.
\begin{figure}[!htb]
	\centering
	\begin{tikzpicture}
		%%%%%%arcos%%%%%%%%%%%%%
		\centerarc[thick,<-](1.5,0.3)(10:170:0.45);
		\centerarc[thick,->](1.5,-0.3)(-10:-170:0.45);
		\centerarc[thick,<-](2.5,0.3)(10:170:0.45);
		\centerarc[thick,->](2.5,-0.3)(-10:-170:0.45);
		\centerarc[thick,->](3.5,-0.3)(-10:-170:0.45);
		\centerarc[thick,<-](3.5,0.3)(10:170:0.45);
		\centerarc[thick,->](4.5,-0.3)(-10:-170:0.45);
		\centerarc[thick,<-](4.5,0.3)(10:170:0.45);
		\centerarc[thick,->](5.5,-0.3)(-10:-170:0.45);
		\centerarc[thick,<-](5.5,0.3)(10:170:0.45);
		%\centerarc[thick,->](6.5,-0.3)(-10:-170:0.45);
		%\centerarc[thick,<-](6.5,0.3)(10:170:0.45);
		%\centerarc[thick,->](7.5,-0.3)(-10:-170:0.45);
		%\centerarc[thick,<-](7.5,0.3)(10:170:0.45);
		
		%%%%%%%%%%%%%%segmento%%%%%%%%%
		\draw (-1,0) -- (8,0);
		
		%%%%%%%%%% particulas em preto %%%%%%%%%%%
		\shade[ball color=black](4,0) circle (0.25);
		
		%%%%%%%%%%% particulas em branco %%%%%%%%%%%
		\filldraw[fill=white, draw=black]
		(6,0) circle (.25)
		(2,0) circle (.25)
		(1,0) circle (.25)
		(0,0) circle (.25)
		(3,0) circle (.25)
		(5,0) circle (.25)
		%(8,0) circle (.25)
		(7,0) circle (.25)
		;

		%%%%%%%%%% rotulos %%%%%%%%%%%%%%
		\draw (1.3,-0.05) node[anchor=north] {\small $\bf - 3 $};
		\draw (2.3,-0.05) node[anchor=north] {\small $\bf - 2 $};
		\draw (3.3,-0.05) node[anchor=north] {\small $\bf -1$};
		\draw (4.3,-0.05) node[anchor=north] {\small $\bf 0$};
		\draw (5.3,-0.05) node[anchor=north] {\small $\bf 1$};
		\draw (6.3,-0.05) node[anchor=north] {\small $\bf 2$};
		\draw (1.5,0.8) node[anchor=south]{$\frac{1}{2}$};
		\draw (1.5,-0.8) node[anchor=north]{$\frac{1}{2}$};
		\draw (2.5,0.8) node[anchor=south]{$\frac{1}{2}$};
		\draw (2.5,-0.8) node[anchor=north]{$\frac{1}{2}$};
		\draw (3.5,0.8) node[anchor=south]{$\frac{\alpha}{2n^{\beta}}$};
		\draw (4.5,-0.8) node[anchor=north]{$\frac{1}{2}$};
		\draw (4.5,0.8) node[anchor=south]{$\frac{1}{2}$};
		\draw (5.5,-0.8) node[anchor=north]{$\frac{1}{2}$};
		\draw (5.5,0.8) node[anchor=south]{$\frac{1}{2}$};
		\draw (3.5,-0.8) node[anchor=north]{$\frac{\alpha}{2n^{\beta}}$};
	\end{tikzpicture}
	\caption{Jump rates for the \textit{slow bond random walk}}
	\label{fig1a}
\end{figure}
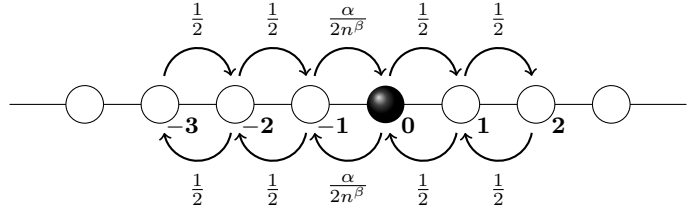
Here $\alpha>0$ and $\beta \ge 0$. Using Kolmogorov’s backward equation, the Feynman--Kac representation, and precise asymptotics for local times of the simple random walk, the authors showed the following trichotomy:  
\begin{itemize}
	\item if $\beta \in [0,1)$, the rescaled walk converges to standard Brownian motion;  
	\item if $\beta > 1$, it converges to reflected Brownian motion;  
	\item if $\beta = 1$, it converges to SNOB.  
\end{itemize}
In all three cases, explicit rates of convergence were established.
The functional central limit theorem proved in the present paper does not include such rates. However, we conjecture that by adapting the methods of \cite{EFS2024}, analogous convergence rates could be derived with relatively minor additional effort.

\subsubsection*{Particle Systems with boundaries}
The study of hydrodynamic limits and fluctuations of interacting particle systems is a classical topic in probability theory. Recently, considerable attention has shifted to systems with boundary effects. Such boundary conditions may arise from reservoirs of particles~\cite{BCG23,MMOL22,Goncalves2019}, or from edges across which particles cross at modified rates. The latter corresponds to ``soft'' boundaries and has been extensively investigated for exclusion processes in
\cite{efgnt,fgn2,fgn3,FN,fgn1,FGSCMP}.

An exclusion process with a single particle reduces to a random walk. Interestingly, as observed in \cite{fgn2,fgn3, fgn1}, analysing the full interacting system may in some situations be simpler than understanding the behaviour of a single particle near a boundary. The present work complements this perspective: together with the techniques developed in~\cite{EFS2024}, it allows one to describe in full generality the boundary behaviour of a single particle at a distinguished edge. Such a description is expected to be a useful ingredient in the study of interacting particle systems with general boundary mechanisms.

\subsection{Outline of the proof}
For the characterisation of the generator of the general Brownian motion on the line and on two half-lines we first distinguish between three cases based on the exponential random variable that determines the exit time of zero (or of $0+$ and $0-$ in the case of two half-lines). A crucial step is the application of Dynkin’s characterization of generators (see \cite[Theorem 3.2.29]{Knight}), which relies on a careful analysis of the exit measure defined in \eqref{eq:exitmeas} (for the case of the line; the argument for two half-lines is analogous). The existence of the corresponding processes then follows from an application of the Hille–Yosida Theorem~\ref{thm:Hille-Yosida}.  

The paper is divided as follows. In Section~\ref{s2} we properly state definitions and results. Section~\ref{s3} is reserved to the most general BM on the real line, and Section~\ref{s4} is reserved to the most general BM on two closed half-lines.

\section{Statements of Results}\label{s2}

\subsection{The most general  BM on \texorpdfstring{$\bb R_{\Delta}$}{R  Delta}}
Consider the state space $\R_\Delta= \bb R\cup\{\Delta\}$, the one-point compactification of $\R$, where $\Delta$ will be called the \textit{cemetery}. We will assume the standard convention  that functions will value zero at  the cemetery $\Delta$. In this setting,  $C(\R_\Delta)$ will be the set of continuous functions such that $f(\Delta)=0$ and $f(x)\to 0$ as $x \rightarrow\pm \infty$. 

\begin{definicao}\label{def:GBMR}
	A stochastic process $(W_t)_{t \geq 0}$ on $\R_\Delta$ is called \textit{a general Brownian motion on $\R_\Delta$ with boundary conditions at the origin} if it satisfies the following properties:
	\begin{itemize}
		\item $(W_t)_{t \geq 0}$  is a strong Markov process with values in $\bb R_{\Delta}$ and it has càdlàg trajectories.
		
		\item The sample paths of $(W_t)_{t \geq 0}$ are continuous on the set $\big\{t\geq 0: \lim_{s \to t^-} W_s \break\mbox{ or } W_t \notin\{0, \Delta\}\big\}$.
		
		\item The point $\Delta$, called the cemetery, is an absorbing state.
		
		\item Let $\tau_0 = \inf\{t\geq 0: W_t = 0\}$ be the hitting time  of $0$. For every initial point $x \in \R$, the law of the process $(W_{t\wedge \tau_0})_{t\geq 0}$ started at $x$ 
		coincides with the law of a standard Brownian motion on $\R$ absorbed at $0$.
	\end{itemize}
\end{definicao}

\begin{observacao}
The above definition ensures that $W$ behaves like a standard Brownian motion until it hits $0$, and has a Markovian behavior afterwards. From the properties of the standard Brownian motion, we know that $\tau_0< \infty$ a.s., and by right-continuity of the paths we have $W_{\tau_0} = 0$. Moreover, the second condition in the previous definition implies that the only allowed jump is from $0$ to $\Delta$, which is an absorbing state. Therefore, $W$ cannot reach $\Delta$ before $0$. 
\end{observacao}
Our first main result consists of the following characterization of the general Brownian motion on $\bb R_\Delta$.
Given a function $g$ we denote $g(0+) = \lim_{x \to 0^+} g(x)$ and $g(0-) = \lim_{x \to 0^-} g(x)$, whenever these limits exist.

\begin{teorema}\label{thm:general_R}
	A stochastic process  is a general Brownian motion $W$ on $\R_\Delta$ with boundary conditions at the origin if, and only if, its infinitesimal generator is given by
	\begin{equation*}
		\msf{L}f(x) = \frac{1}{2}f''(x) \text{ for } x \neq 0\,, \quad   \msf{L}f(0) = \frac{1}{2}f''(0-) = \frac{1}{2}f''(0+)\,,\quad \text{ and }\quad
		\msf{L}f(\Delta)=0\,,
	\end{equation*}
	with domain $\mf {D}(\msf{L})$ consisting of all functions $f$ belonging to $C(\R_{\Delta})$, which are twice differentiable in $\R \backslash  \{0\}$ with second derivative extending to an element of $C(\R_{\Delta})$, and the equation
	\begin{align}
		\label{firstdomain}
		& \nonumber c_1f(0) + c^-_2f'(0-) - c^+_2f'(0+) + \frac{c_3}{2}f''(0+) \;=\;0
	\end{align}
	holds for some constants $c_1, c^-_2, c^+_2, c_3 \geq 0$ with $c_1 + c^-_2 + c^+_2 + c_3 = 1$, and $c_1 \neq 1$.
\end{teorema}
\begin{observacao}
Note that by Remark~\ref{rem:dif} below if $f$ is in the domain of $\msf{L}$, then its first derivative has side limits in zero. 
\end{observacao}
\begin{observacao}
	Observe that if $c_1 =1$, then since $c_1 + c^-_2 + c^+_2 + c_3 = 1$, one would have the boundary condition $f(0) = 0$. However, this set of functions is not dense in $C(\R_\Delta)$, and therefore it cannot be the domain of a generator (see Remark~\ref{obs:Feller} and the Hille-Yosida Theorem~\ref{thm:Hille-Yosida} in the appendix). This explains why the case $c_1= 1$ is excluded in the above result.
\end{observacao}
\begin{observacao}
We note that by~\cite[Lemma 3.2.28 ii]{Knight}, the identity $(\msf{L}f)(\Delta)=0$ for all $f\in \mf {D}(\msf{L})$ implies that $\Delta$ is an absorbing point.
\end{observacao}

Consider a general Brownian motion on $\R_\Delta$ such that its constants $c_1, c^-_2, c^+_2, c_3$ are positive and  assume that $c_2^- + c_2^+>0$. Dividing the domain's equation by $c_2^- + c_2^+$, we get that
\begin{equation*}
	\frac{c_1}{c_2^- + c_2^+}f(0) + \frac{c_2^-}{c_2^- + c_2^+}f'(0-) - \frac{c_2^+}{c_2^- + c_2^+}f'(0+) + \frac{c_3}{2(c_2^- + c_2^+)}f''(0+) \;=\; 0\,.
\end{equation*}
Setting $\gamma = \frac{c_1}{c_2^-+ c_2^+}$, $\beta = \frac{c_2^+}{c_2^-+ c_2^+}$, $c = \frac{c_3}{2(c_2^- + c_2^+)}$, we have that
\begin{equation*}
	cf''(0+) \;=\; \beta f'(0+) - (1 - \beta)f'(0-) - \gamma f(0)\,.
\end{equation*}
This corresponds to the boundary condition of the Brownian Motion which is skew at $0$, sticky at $0$ and killed elastically at $0$, see Borodin and Salminen's book \cite[page~127, Section~13, Appendix~1]{Borodin}. Hence, the above result implies that the general Brownian motion on $\R_\Delta$ with boundary conditions at the origin essentially coincides with the Skew Sticky Killed at Zero  Brownian Motion and no new Brownian-type process can be extracted in this context.
The case $c_2^-+c_2^+=0$ corresponds to the exponential holding Brownian motion. Once it reaches zero it stays there for an exponential amount of time until gets killed, i.e., it jumps to the cemetery $\Delta$ and stays there forever. Thus, if it starts on the positive (respectively negative) half-line it is rather a process on $\{\Delta\}\cup[0,\infty)$ (respectively $\{\Delta\}\cup(-\infty,0]$).

\subsection{The most general  BM on \texorpdfstring{$\bb G_{\Delta}$}{G Delta}}
Our second main result consists of a similar result on the state space $\bb G_\Delta = \bb G \cup \{\Delta\}$ which is the one-point compactification of the set $\bb G= (-\infty, 0-] \cup [0+, \infty)$  composed of two closed half-lines. Here $\Delta$ is the compactification point. We highlight  that $0-$ and $0+$ are distinct points.
Similarly to before, the set $C(\bb G_\Delta)$ will be the set of continuous functions $\function{f}{\bb G_\Delta}{\R}$ satisfying $f(\Delta)=0$, and having zero limit as $x \rightarrow\pm \infty$. 

\begin{definicao}
	\label{BMonE}
	A stochastic process $(W_t)_{t \geq 0}$ on $\bb G_\Delta$ is called a general Brownian motion on $\bb G_\Delta$ with boundary conditions at the origin if it satisfies the following properties:
	\begin{itemize}
		\item $(W_t)_{t \geq 0}$  is a strong Markov process with values in $\bb G_\Delta$ and has càdlàg trajectories.
		
		\item The sample paths of $(W_t)_{t \geq 0}$ are continuous on the set $\big\{t\geq 0: \lim_{s \to t^-} W_s\linebreak \mbox{ or } W_t \notin\{0+, 0-, \Delta\}\big\}$.
		
		\item The point $\Delta$, called the cemetery, is an absorbing state.
		
		\item Let $\tau_{0+} = \inf\{t\geq 0: W_t = 0+\}$ be the hitting time  of $0+$. For every initial point $x \in [0+, \infty)$, the law of the process $(W_{t\wedge \tau_{0+}})_t$  
		coincides with the law of a standard Brownian motion on $[0, \infty)$ absorbed at $0$. Here, we are identifying $0+$ with $0$. Similarly, for every starting point $x \in (-\infty, 0-]$, the law of $(W_{t\wedge \tau_{0-}})_t$  coincides with that of a Brownian motion on $(-\infty, 0]$ absorbed at $0$. Here, we are identifying $0-$ with $0$ and $\tau_{0-} = \inf\{t\geq 0: W_t = 0-\}$ is the hitting time of $0-$.
	\end{itemize}
\end{definicao}

\begin{teorema}\label{teo:5.5}
	A stochastic process on $\bb G_\Delta$ is a general Brownian motion on $\bb G_\Delta$ with boundary conditions at the origin if and only if its infinitesimal generator is given by
		\begin{align*}
		&\msf{L}f(x) = \frac{1}{2}f''(x)\, \text{ for } x \in \bb G \quad \text{ and }\quad \msf{L}f(\Delta)=0\,,
		\end{align*}
	and its domain $\mf {D}(\msf{L})$ coincides with the set of functions $f \in C(\bb G_\Delta)$ twice differentiable with $f''\in C(\bb G_\Delta)$ such that the following two equations hold
\begin{align*}
	c_1^+f(0+) + a^+\big(f(0+) - f(0-)\big) - c_2^+f'(0+) + \frac{c_3^+}{2}f''(0+)	& \;=\; 0\,, \quad \text{ and } \\ 
	c_1^-f(0-) + a^-\big(f(0-) - f(0+)\big) + c_2^-f'(0-) + \frac{c_3^-}{2}f''(0-) & \;=\; 0\,,
\end{align*}
for some nonnegative constants $a^+, a^-, c_i^+, c_i^-$, $i = 1, 2, 3$. Moreover, $\max\{c_2^+, c_3^+\}>0$ and $\max\{c_2^-, c_3^-\}>0$.
\end{teorema}

	The coefficients appearing in the definition of $\mf{D}(\msf{L})$ in the statement of Theorem~\ref{teo:5.5} admit the following interpretation.
	The coefficient $c_1^+$ (respectively $c_1^-$) is associated with the rate at which the process jumps from $0+$ (respectively $0-$) to the cemetery; that is, it represents a \emph{killing rate}.
	The coefficient $a^+$ (respectively $a^-$) corresponds to the rate at which the process jumps from $0+$ to $0-$ (respectively from $0-$ to $0+$); these are therefore the rates governing the transitions between the two half-lines and consequently determine the skewness of the process.
	The coefficient $c_2^+$ (respectively $c_2^-$) quantifies the reflection strength at $0+$ (respectively $0-$), while $c_3^+$ (respectively $c_3^-$) characterizes the degree of stickiness of the process at $0+$ (respectively $0-$).
	
	It is worth noting that the equations defining the domain of the generator in Theorem~\ref{teo:5.5} are homogeneous and may thus be normalized. This observation is significant, as it implies that stickiness at $0+$ (respectively $0-$) and reflection at $0+$ (respectively $0-$) are not independent phenomena but rather competing effects.
	
	According to \cite[Proposition~1]{Lejay}, the infinitesimal generator of the SNOB on $\bb G$ is given by $\msf{L}f = \frac{1}{2}f''$, with domain $\mf{D}(\msf{L})$ consisting of functions satisfying $f'' \in C(\bb G_\Delta)$  with
	\begin{align*}
		f'(0+) \;=\; f'(0-) \;=\; \frac{\kappa}{2}\big(f(0+) - f(0-)\big)
	\end{align*}
	where $\kappa$ is a positive constant.
	This case corresponds to a particular instance of the family of processes described in Theorem~\ref{teo:5.5}, obtained by setting $a^+ = a^- = \kappa/2$, $c_2^+ = c_2^- = 1$, and $c_1^+ = c_1^- = c_3^+ = c_3^- = 0$.
	In particular, the condition $a^+ = a^-$ implies that the switching between half-lines occurs symmetrically. Allowing instead $a^+ \neq a^-$ yields a Snapping Out Brownian Motion with \emph{skewness}.
	
	The above considerations motivate us to refer to the class of processes obtained in Theorem~\ref{teo:5.5} as the \emph{Skew Sticky Killed Snapping Out Brownian Motion}.
	This constitutes a novel Brownian-type process, which, as demonstrated herein, represents the most general form of Brownian motion on the state space $\bb G_\Delta$.

\begin{observacao}
		\label{obs:Feller}
	We note that the general Brownian motions on $\bb R_\Delta$ and on $\bb G_\Delta$ are defined as strong Markov processes with càdlàg sample paths. In the proofs of Theorems~\ref{thm:general_R} and~\ref{teo:5.5}, we show that, in the present setting, these assumptions imply the Feller property. Conversely, every Feller process admits a càdlàg modification (see \cite[Theorem~6.15]{leGall}). Therefore, in this context, the strong Markov property together with càdlàg paths is equivalent to the Feller property. 
\end{observacao}

\begin{observacao}
Note that if $c_2^+=c_3^+=0$, then the operator defined in Theorem~\ref{teo:5.5} would not define a generator of a Feller process. Indeed, if $c_2^+=c_3^+=0$, then $f(0+)= \tfrac{a^+}{c_1^++a^+}f(0-)$. Hence, the values of $f(0+)$ and $f(0-)$ would be coupled and therefore the domain of $\msf{L}$ would not be dense in $C(\bb G_\Delta)$. By the Hille-Yosida Theorem~\ref{thm:Hille-Yosida} in the appendix $\msf{L}$ would not be the generator of a Feller process. The same observation explains why one must demand that $\max\{c_2^-, c_3^-\}>0$.
\end{observacao}

\subsection{Functional CLT towards the most general BMs}
As a consequence of the precise characterization of the generator in Theorem~\ref{teo:5.5}, we can easily obtain  functional CLTs towards the  most general BMs on $\bb R_\Delta$ and $\bb G_\Delta$. 

\subsubsection{A functional CLT towards the most general BM on \texorpdfstring{$\bb R_\Delta$}{R Delta}}
\begin{figure}[!htb]
	\centering
	\begin{tikzpicture}[scale=1]
		\draw (-5,0)--(5,0);
		
		\draw (2.5,0.6) node[above]{$\displaystyle\frac{1}{2}$};
		\draw (3.5,0.6) node[above]{$\displaystyle\frac{1}{2}$};
		\draw (0.5,0.6) node[above]{$\displaystyle\frac{B_+}{n}$};
		\draw (-0.5,0.6) node[above]{$\displaystyle\frac{B_-}{n}$};

		\centerarc[thick,<-](0.5,0)(35:140:0.6);
		\centerarc[thick,<-](-0.5,0)(140:40:0.6);
		
		\centerarc[thick,->](2.5,0)(40:140:0.6);
		\centerarc[thick,->](3.5,0)(140:40:0.6);

		%		\centerarc[thick,<-](0,1)(-55:-125:1.5);	
		%		\centerarc[thick,->](0,-1)(55:125:1.5);

		\filldraw[fill=white] (0,-2.2) circle (.15);
		\draw (0.1,-2.2) node[right]{$\Delta$};
		
		\draw (0,-0.2) node[below]{$\frac{0}{n}$};
		\draw (1,-0.2) node[below]{$\frac{1}{n}$};
		\draw (2,-0.2) node[below]{$\frac{2}{n}$};
		\draw (3,-0.2) node[below]{$\frac{3}{n}$};
		\draw (4,-0.2) node[below]{$\frac{4}{n}$};
		\draw (5.7,0) node[left]{...};

		\draw (-1,-0.2) node[below]{$\frac{-1}{n}$};
		\draw (-2,-0.2) node[below]{$\frac{-2}{n}$};
		\draw (-3,-0.2) node[below]{$\frac{-3}{n}$};
		\draw (-4,-0.2) node[below]{$\frac{-4}{n}$};
		\draw (-5.1,0) node[left]{...};

		\filldraw[fill=white] (0,0) circle (.15);
		\filldraw[fill=white] (1,0) circle (.15);
		\filldraw[fill=white] (2,0) circle (.15);
		\filldraw[ball color=black] (3,0) circle (.15);
		\filldraw[fill=white] (4,0) circle (.15);

		\filldraw[fill=white] (-1,0) circle (.15);
		\filldraw[fill=white] (-2,0) circle (.15);
		\filldraw[fill=white] (-3,0) circle (.15);
		\filldraw[fill=white] (-4,0) circle (.15);
		
		\draw[very thick,->] (0,-0.9) -- node [midway, anchor = west]{$\displaystyle\frac{A}{n^2}$} (0,-1.9);
		
	\end{tikzpicture}
	\caption{The jump rates for the \textit{boundary random walk} on $\bb R_{n, \Delta}$ are $n^2$ times the ones shown in the picture.}
	\label{fig:RW_1}
\end{figure}
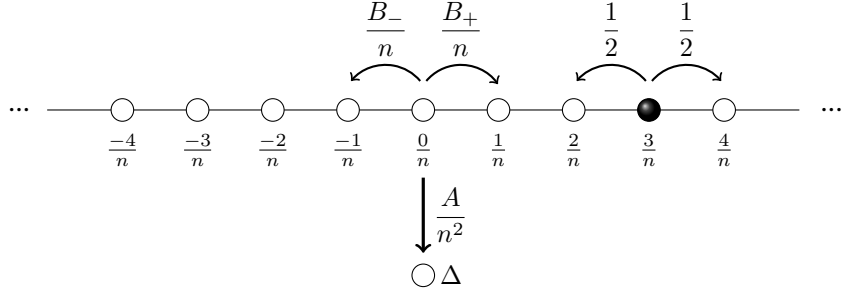

Consider the continuous time random walk on the state space
$\bb R_{n, \Delta} := \{\ldots, \frac{-2}{n}, \frac{-1}{n}, \frac{0}{n}, \frac{1}{n}, \frac{2}{n}, \ldots\} \cup \{\Delta\}$ whose jump rates are described in Figure~\ref{fig:RW_1}.
That is, its generator $\msf{L}_n$  is  given by
\begin{equation}
	\label{RWgm_R}
	\begin{split}
		&\msf{L}_nf(x)=\\
		&
		n^2\times \begin{cases}
			\!\dfrac{1}{2}\Big[f\big(x\!+\!\pfrac{1}{n}\big)\!-\!f(x)\Big]+\dfrac{1}{2}\Big[f\big(x\!-\!\pfrac{1}{n}\big)\!-\! f(x)\Big], \quad  x \notin \{\Delta, \frac{0}{n}\}  \vspace{5pt}\\
			\!\dfrac{A}{n^2}\Big[f\big(\Delta\big)\!-\!f(\pfrac{0}{n})\Big]\!+\!\dfrac{B_+}{n}\Big[f\big(\pfrac{1}{n}\big)\!-\!f\big(\pfrac{0}{n}\big)\Big]
			\!+\!\dfrac{B_-}{n^2}\Big[f\big(\pfrac{-1}{n}\big)\!-\!f\big(\pfrac{0}{n}\big)\Big],  x\!=\! \frac{0}{n}\vspace{5pt}\\
			0,\quad x =\Delta
		\end{cases}
	\end{split}
\end{equation}
with domain  $\mf D(\msf{L}_n) = \big\{f: \bb R_{n, \Delta} \to \bb R \text{ s.t. }\lim_{k\to \pm \infty} f(k) = 0 \text{ and } f(\Delta)=0\big\}$.
\begin{teorema}\label{TCL_R}
	Fix $u\in \bb R\backslash \{0\}$ and let $\{X_n(t): t\geq 0\}$ be the boundary random walk whose generator is $\msf L_n$ of parameters $A, B_+, B_-\geq 0$, starting from the point $\frac{\lfloor un \rfloor}{n}\in \bb R_{n,\Delta}\subset \bb R_{\Delta}$. Then the sequence of processes $\{X_n(t): t\geq 0\}$ converges weakly to  $\{X(t):t\geq 0\}$ in the Skorohod topology of $\msf{D}_{\bb{R}_{\Delta}}[0,\infty)$, where $X$ is the \textit{Skew Sticky Killed  Brownian Motion} on $\bb{R}_{\Delta}$ characterized in Theorem~\ref{thm:general_R}, starting from the point $u$ and whose  parameters are 
	\begin{align}\label{eq: gerador_R}
		 c_1= A\,, \quad c^+_2 = B^+\,, \quad  c^-_2 = B_-\,, \quad c_3 = 1\,. 
	\end{align}
\end{teorema}

\subsubsection{A functional CLT towards the most general BM on \texorpdfstring{$\bb G_\Delta$}{G Delta}}
 Consider the continuous time random walk on the state space
$\bb G_{n, \Delta} := \{\ldots, \frac{-2}{n}, \frac{-1}{n}, \frac{0-}{n}\} \cup \{\Delta\} \cup \{\frac{0+}{n}, \frac{1}{n}, \frac{1}{n}, \ldots\}$ whose jump rates are described in Figure~\ref{fig:RW}.
\begin{figure}[!htb]
	\centering
	\begin{tikzpicture}[scale=1]
		\draw (1,0)--(5,0);
		\draw (-1,0)--(-5,0);
		
		\draw (3.5,0.6) node[above]{$\displaystyle\frac{1}{2}$};
		\draw (4.5,0.6) node[above]{$\displaystyle\frac{1}{2}$};
		\draw (1.5,0.6) node[above]{$\displaystyle\frac{B_+}{n}$};
		\draw (-1.5,0.6) node[above]{$\displaystyle\frac{B_-}{n}$};
		\draw (0,0.5) node[above]{$\displaystyle\frac{C_+}{n^{2}}$};
		\draw (0,-0.5) node[below]{$\displaystyle\frac{C_-}{n^{2}}$};
		
		\centerarc[thick,<-](1.5,0)(30:140:0.6);
		\centerarc[thick,->](-1.5,0)(30:140:0.6);

		\centerarc[thick,<-](0,1)(-55:-125:1.5);	
		\centerarc[thick,->](0,-1)(55:125:1.5);
		
		\centerarc[thick,->](3.5,0)(40:150:0.6);
		\centerarc[thick,<-](4.5,0)(30:140:0.6);
		
		\filldraw[fill=white] (0,-2.2) circle (.15);
		\draw (0,-2.4) node[below]{$\Delta$};
		
		\draw (1,-0.2) node[below]{$\frac{0+}{n}$};
		\draw (2,-0.2) node[below]{$\frac{1}{n}$};
		\draw (3,-0.2) node[below]{$\frac{2}{n}$};
		\draw (4,-0.2) node[below]{$\frac{3}{n}$};
		\draw (5,-0.2) node[below]{$\frac{4}{n}$};
		\draw (5.7,0) node[left]{...};
		
		\draw (-1,-0.2) node[below]{$\frac{0-}{n}$};
		\draw (-2,-0.2) node[below]{$\frac{-1}{n}$};
		\draw (-3,-0.2) node[below]{$\frac{-2}{n}$};
		\draw (-4,-0.2) node[below]{$\frac{-3}{n}$};
		\draw (-5,-0.2) node[below]{$\frac{-4}{n}$};
		\draw (-5.1,0) node[left]{...};

		\filldraw[fill=white] (1,0) circle (.15);
		\filldraw[fill=white] (2,0) circle (.15);
		\filldraw[fill=white] (3,0) circle (.15);
		\filldraw[ball color=black] (4,0) circle (.15);
		\filldraw[fill=white] (5,0) circle (.15);
		
		\filldraw[fill=white] (-1,0) circle (.15);
		\filldraw[fill=white] (-2,0) circle (.15);
		\filldraw[fill=white] (-3,0) circle (.15);
		\filldraw[fill=white] (-4,0) circle (.15);
		\filldraw[fill=white] (-5,0) circle (.15);
		
		\draw[thick,->] (1,-0.9) -- node [midway, anchor = north west]{$\displaystyle\frac{A_+}{n^2}$} (0.1,-2);
		\draw[thick,->] (-1,-0.9) -- node [midway, anchor = north east]{$\displaystyle\frac{A_-}{n^2}$} (-0.1,-2);
	\end{tikzpicture}
	\caption{The jump rates for the \textit{boundary random walk} on $\bb G_{n, \Delta}$ are $n^2$ times the ones shown in the picture.}
	\label{fig:RW}
\end{figure}
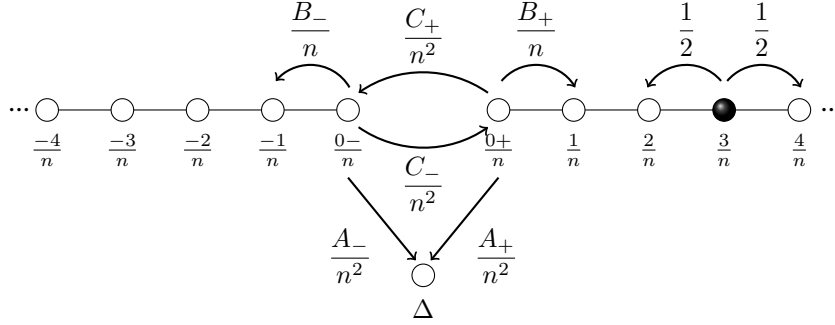

That is, its generator $\msf{L}_n$  is  given by
\begin{equation}
	\label{RWgm}
	\begin{split}
			&\msf{L}_nf(x)\;=\;\\
		&
		n^2\times \begin{cases}
			\!\dfrac{1}{2}\Big[f\big(x\!+\!\pfrac{1}{n}\big)\!-\!f(x)\Big]+\dfrac{1}{2}\Big[f\big(x\!-\!\pfrac{1}{n}\big)\!-\! f(x)\Big], \quad  x \notin \{\Delta, \frac{0+}{n}, \frac{0-}{n}\}  \vspace{5pt}\\
			\!\dfrac{A_+}{n^2}\Big[f\big(\Delta\big)\!-\!f(\pfrac{0+}{n})\Big]\!+\!\dfrac{B_+}{n}\Big[f\big(\pfrac{1}{n}\big)\!-\!f\big(\pfrac{0+}{n}\big)\Big]
			\!+\!\dfrac{C_+}{n^2}\Big[f\big(\pfrac{0-}{n}\big)\!-\!f\big(\pfrac{0+}{n}\big)\Big],  x\!=\! \frac{0+}{n}\vspace{5pt}\\
			\!\dfrac{A_-}{n^2}\Big[f\big(\Delta\big)\!-\!f(\pfrac{0-}{n})\Big]\!+\!\dfrac{B_-}{n}\Big[f\big(\pfrac{-1}{n}\big)\!-\!f\big(\pfrac{0-}{n}\big)\Big]
			\!+\!\dfrac{C_-}{n^2}\Big[f\big(\pfrac{0+}{n}\big)\!-\!f\big(\pfrac{0-}{n}\big)\Big],  x\!=\! \frac{0-}{n}\vspace{5pt}\\
			0,\quad x =\Delta
		\end{cases}
	\end{split}
\end{equation}
with domain  $\mf D(\msf{L}_n) = \big\{f: \bb G_{n, \Delta} \to \bb R \text{ s.t. }\lim_{k\to \pm \infty} f(k) = 0 \text{ and } f(\Delta)=0\big\}$.
\begin{teorema}\label{TCL}
	Fix $u\in \bb G\backslash \{0+, 0-, \Delta\}$ and let $\{X_n(t): t\geq 0\}$ be the boundary random walk whose generator is $\msf L_n$ of parameters $A_+, A_-, B_+, B_-, C_+, C_-\geq 0$, starting from the point $\frac{\lfloor un \rfloor}{n}\in \bb G_{n,\Delta}\subset \bb G_{\Delta}$. Then the sequence of processes $\{X_n(t): t\geq 0\}$ converges weakly to  $\{X(t):t\geq 0\}$ in the Skorohod topology of $\msf{D}_{\bb{G}_{\Delta}}[0,\infty)$, where $X$ is the \textit{Skew Sticky Killed Snapping Out Brownian Motion} on $\bb{G}_{\Delta}$ characterized in Theorem~\ref{teo:5.5}, starting from the point $u$ and whose  parameters are 
	\begin{align}
		&c_1^+ = A_+\,,\quad  c_2^+ = B_+\quad  c_3^+ = 1\,,\quad  a^+ = C_+ \label{eq: rel_1} \\
		&c_1^- = A_-\,,\quad  c_2^- = B_-\quad  c_3^- = 1\,,\quad  a^- = C_-\,. 
	\end{align}
\end{teorema}
\begin{figure}[!htb]
	\centering
	\includegraphics[width=1\linewidth]{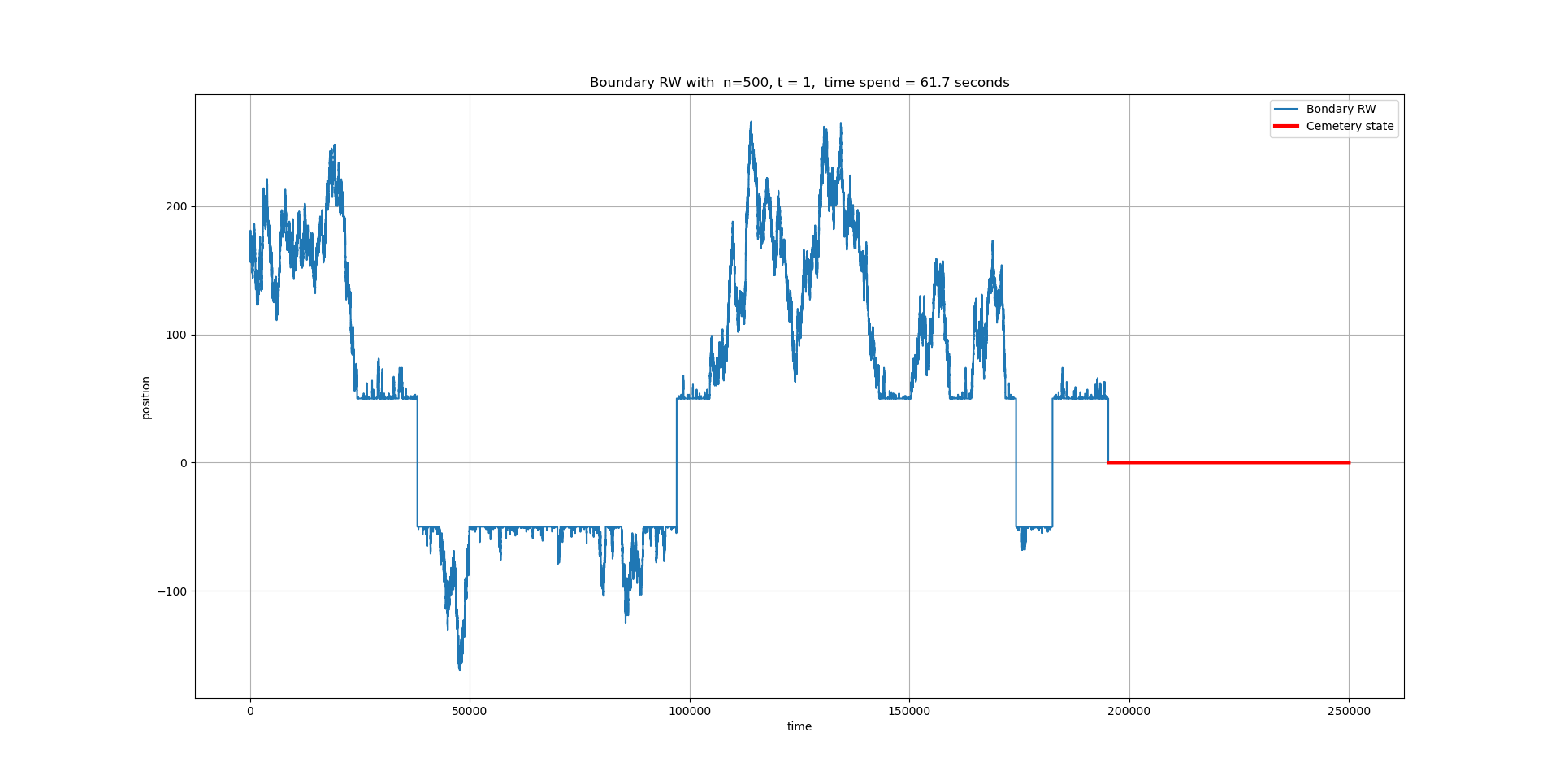}
	\caption{Simulation of the Boundary Random Walk with parameters $A_+ = 0.25$, $A_- = 0.25$, $B_+ = 2$, $B_- = 2$, $C_+ = 6$, and 
		$C_- = 4$. The macroscopic time is $t = 1$, the macroscopic initial position is $u = 1/3$ and the discrete parameter is $n = 500$. Note that the walk switches between half-lines and eventually goes to the cemetery. This is highlighted in red. Although not so evident from the picture, since $C_+> C_-$, the walk has a preference to stay in the negative half-line. The stickiness at the points $0+$ and $0-$ is visible.}
	\label{fig:figure1}
\end{figure}

Similar results to those of \cite{EFJP} can be achieved on Boundary Random Walk defined in \eqref{RWgm}, including Berry-Essen estimates and  more general jump rates. Since this is not the main topic of this paper, we do not enter into details. In Figure~\ref{fig:figure1} we exhibit  a computational simulation of the Boundary RW. The code written to produce this picture is available in  \cite{Two_Half}.
    
\section{The most general BM on \texorpdfstring{$\R_{\Delta}$}{R Delta} -- Proof of Theorem~\ref{thm:general_R}}\label{s3}

 We start by recalling a key property of strong Markov processes. To that end, we use the conventions that an exponentially distributed random variable $X$ with parameter zero is almost surely equal to infinity and an exponentially distributed random variable $Y$ with parameter infinity is almost surely equal to zero.
\begin{proposicao}[see Propositions~III.2.19 and III.3.13 of \cite{RevuzYor}]\label{prop:jump time}
	Let $\mc W$ be a strong Markov process on the Polish space $E$. Denote by $P_x$ the law of the process starting from $x\in E$
and let $\sigma_x = \inf  \{t > 0 : \mc W_t \neq x\}$ be the waiting time of $x$. Then there exists $\lambda(x) \in [0, \infty]$ such that $\sigma_x$ is exponentially distributed with parameter $\lambda(x)$ under $P_x$. Moreover, if $0 < \lambda(x) < \infty$, then $P_x(\mc W_{\sigma_x} \neq x) = 1$.
\end{proposicao}
Note that in the case $0 < \lambda(x) < \infty$, the statement above implies that the process $\mc W$ must leave $x$ by a jump.\bigskip

In this section we will  denote by $W$ a general Brownian motion on $\R_{\Delta}$ with boundary conditions at the origin.
By Proposition~\ref{prop:jump time}, we have that $T = \inf \{t > 0 : W_t \neq 0\}$ is exponentially distributed with parameter $\lambda \in [0, \infty]$ under $P_0$. We distinguish between three cases.\bigskip

 \textbf{Case 1:} $0 < \lambda < \infty$. 	
	Due to Proposition~\ref{prop:jump time} the process leaves $0$ by a jump, so the continuity of paths on $\R$ forces $W$ to jump to $\Delta$. Hence, $W_t = \Delta$ for all $t \geq T$ in this case.
	
 \textbf{Case 2:} $\lambda = \infty$.	
	Here, $P_0(T = 0) = 1$, so the process leaves $0$ at once. Since $P_0(W_0 = 0) = 1$ and the paths are right-continuous, the process cannot immediately jump from $0$ to $\Delta$, otherwise the process would be left-continuous. We thus conclude that the process leaves zero immediately, but not to $\Delta$.
	
\textbf{Case 3:} $\lambda = 0$.	
	In this situation, we have $T = \infty$ a.s., which means that $0$ is a trap.\bigskip

We now prove in the next proposition that a general Brownian motion on $\R_\Delta$ with boundary conditions at the origin is a Feller process on $C(\R_{\Delta})$. In the proof, we will write $\tau^1$ for the first jump of the paths of $W_t$. In particular, $\tau^1 = \tau_{\Delta} = \inf \{s \geq 0 : W_s = \Delta\}$. Some arguments below could be slightly simplified noting that $f(W_t)$ vanishes for $f\in C(\R_{\Delta})$, if $\tau^1 \leq t$ (more precisely, the third term in Equation~\eqref{eq:example}). However, we do not do this in order to make the arguments easier to adapt to the setting of Section~\ref{section:BMG}. 
\begin{proposicao}
	\label{BMFeller}
	Every general Brownian motion $W$ on $\R_{\Delta}$ with boundary conditions at the origin has a Feller semigroup on $C(\R_{\Delta})$.
\end{proposicao}

\begin{proof}
	Since $W$ is a Markov process, its semigroup is given by $Q_tf(x) = E_x[f(W_t)]$, for  $f \in C(\R_{\Delta})$, $x \in \R_{\Delta}$, and $t \geq 0$. To see that it is Feller, it suffices to verify that $Q_tf\in C(\R_{\Delta})$ for all $f \in C(\R_{\Delta})$, and the pointwise convergence
	\begin{equation*}
		|E_x[f(W_t)] - f(x)| \longrightarrow 0 \mbox{ as } t \searrow 0,
	\end{equation*}
	for all $x \in \R_{\Delta}$ and $f \in C(\R_{\Delta})$, see for instance \cite[Propositions~III.2.4]{RevuzYor}. The latter follows directly by Dominated Convergence Theorem, due to the right-continuity of the paths and the fact that $f \in C(\R_{\Delta})$ is bounded. We show now that $Q_t f\in C(\R_{\Delta})$.
	First we prove that for $t \geq 0$ and $f \in C(\R_{\Delta})$,
	\begin{equation}
		\label{eq:decay}
		\lim_{x \to \infty} E_x[f(W_t)] \;=\; 0\,.
	\end{equation}
 Given $\eps > 0$ there exists $M > 0$ such that $|f(x)|< \eps$, for all $x > M$. Since we are interested in the limit as $x\to\infty$, assume without loss of generality that  $x > M$. Thus,
	\begin{align}
		\label{eq:example}
		\big|E_x[f(W_t)]\big|   \leq &\; E_x\big[|f(W_t)|\mathbf{1}_{\{W_t > M\} \cap \{\tau^1 > t\}}\big] + E_x\big[|f(W_t)|\mathbf{1}_{\{W_t \leq M\} \cap \{\tau^1 > t\}}\big]\\ \nonumber
		&\; + E_x\big[|f(W_t)|\mathbf{1}_{\{\tau^1 \leq t\}}|\big] \\ \nonumber
		\leq & \;\eps + \Vert f\Vert_\infty P_x(W_t \leq M, \tau^1 > t) + \Vert f\Vert_\infty P_x(\tau^1 \leq t).
	\end{align}
	Therefore, it is sufficient to show that for every $t \geq 0$ and $M > 0$,
	\begin{equation}\label{eq:limits}
		\lim_{x \rightarrow \infty} P_x(W_t \leq M,\tau^1 > t)  = 0 \quad \text{ and }\quad  \lim_{x \rightarrow \infty} P_x(\tau^1 \leq t)  = 0.
	\end{equation}
	To do so, let $(B_t)_{t \geq 0}$ be a standard Brownian motion and fix $t \geq 0$ and $M > 0$. Consider the hitting times  
	\begin{equation*}
		\tau_x = \inf \{s \geq 0 :  W_s = x\}\quad  \text{ and } \quad \tau_x^B = \inf \{s \geq 0 : B_s = x\}\,.
	\end{equation*}

By the continuity of the paths, since $x > M$, we have $P_x(\tau_0 > \tau_M) = 1$. In particular, Definition~\ref{def:GBMR} implies that $\tau_M \overset{d}{=} \tau_M^B$. Therefore
	\begin{align*}
		P_x(W_t \leq M, \tau^1 > t) &\leq P_x(\tau_M \leq t) = P_x^B(\tau_M^B \leq t) = P_0^B(\tau_{x-M}^B \leq t),
	\end{align*}
	where $P^B_x$ denotes the probability relative to a standard Brownian motion $B$ starting at $x$. From now on, for simplicity, we will use the same notation $P_x$ for both probabilities. 
	From the Reflection Principle (see \cite[Corollary 2.22]{leGall} for instance), we know that
	\begin{equation}
		\label{eq: reflection}
		P_0(\tau^B_{x-M} \leq t) = P_0(|B_t| \geq x-M) = 2P_0(B_t \geq x-M).
	\end{equation}
	Therefore,
	\begin{align*}
		P_x(W_t \leq M, \tau^1 > t) &\leq P_x(\tau_M^B \leq t) = 2P_0(B_t \geq x - M) \\
		& = 2\int_{x-M}^{\infty} \frac{e^{-\frac{y^2}{2t}}}{\sqrt{2\pi t}} \, dy \longrightarrow 0 \; \mbox{ as } x \rightarrow \infty,
	\end{align*}	
	so we get the first limit in \eqref{eq:limits}. For the  second limit in \eqref{eq:limits}, recall that a general Brownian motion starting from $x \neq 0$ does not jump to $\Delta$ without passing through $0$ and it is equal in distribution to a standard Brownian motion in the time interval $[0, \tau_0]$. Then, 
	\begin{equation*}
		P_x(\tau^1 \leq t) \leq P_x(\tau^B_0 \leq t) \longrightarrow 0 \;\mbox{ as } x \rightarrow \infty,
	\end{equation*}
	again by the Reflection Principle. By symmetry, the same argument shows that $E_x[f(W_t)]$ goes to $0$ as $x \rightarrow -\infty$.
	
	Finally, we will show that $E_x[f(W_t)]$ is continuous in any $x \in \R_{\Delta}$, for any $t \geq 0$. The limit in \eqref{eq:decay} already shows the continuity at $\Delta$, so we may assume $x \in \R$.
	
	We will construct a coupling of the family of processes $(W_t^y)_{t\geq 0}$ indexed by their starting points $y \in \R$. Let $(B_t)_{t \geq 0}$ be a Brownian motion starting from $0$ independent of $W^0_t$, a general Brownian motion with boundary conditions at the origin, also starting from zero. We define
	\begin{equation}
		\label{eq: coupling}
		W^y_t \;=\; \begin{cases}
			y + B_t\,, & \text{ if } 0 \leq t \leq \tau_{-y}^B \vspace{3pt}\\
			W^0_{t - \tau_{-y}^B} \,, & \text{ if } t \geq  \tau_{-y}^B\,.
		\end{cases}
	\end{equation}
	Note that when $B_t$ hits $-y$ for the first time, the process $W_t^y$  hits $0$ for the first time.
	In this way, up to  time $\tau_{-y}^B$, the process $W_t^y$ follow the same Brownian motion, but each one starting from its respective initial point $y$. From time $\tau_{-y}^B$ on, the processes $W_t^y$ follow  $W^0_t$, $t\geq 0$.
	
	By symmetry, we can assume without loss of generality that $x \in [0, \infty)$. We first focus on the case $x \in (0, \infty)$. For $t \geq 0$, and $y \in \R$ we have that
	\begin{align*}
		&\big|E[f(W_t^y)] - E[f(W_t^x)]\big| \\ &\leq E\big[|f(W_t^y) - f(W_t^x)|\mathbf{1}_{\{t < \tau_{-y}^B \wedge \tau_{-x}^B\}}\big] + E\big[|f(W_t^y) - f(W_t^x)|\mathbf{1}_{\{t \geq \tau_{-y}^B \wedge \tau_{-x}^B\}}\big] \\
		&\leq E\big[|f(y + B_t) - f(x + B_t)|\big] + E\big[|f(W_t^y) - f(W_t^x)|\mathbf{1}_{\{t \geq \tau_{-y}^B \wedge \tau_{-x}^B\}}\big]\,.
	\end{align*}
	By the uniform continuity of $f$, we have $E[|f(y + B_t) - f(x + B_t)|] \longrightarrow 0$ as $y \rightarrow x$. For the remaining term,  suppose without loss of generality  that $0 < x < y$. Thus, $\tau_{-x}^B < \tau_{-y}^B$, and
	\begin{align*}
		& E\big[|f(W_t^y) - f(W_t^x)|\mathbf{1}_{\{t \geq \tau_{-y}^B \wedge \tau_{-x}^B\}}\big]\\
		&= E\big[|f(W_t^y) - f(W_t^x)|\mathbf{1}_{\{\tau_{-x}^B \leq t \leq \tau_{-y}^B\}}\big] + E\big[|f(W_t^y) - f(W_t^x)|\mathbf{1}_{\{\tau_{-y}^B < t\}}\big] \\
		&\leq 2\Vert f\Vert_\infty P(\tau_{-x}^B \leq t \leq \tau_{-y}^B) + E\big[|f(W^0_{t - \tau_{-y}^B}) - f(W^0_{t - \tau_{-x}^B})|\mathbf{1}_{\{\tau_{-y}^B < t\}}\big]\,.
	\end{align*}
	Due to the fact that almost surely $\lim_{y \to x} \tau_{-y}^B = \tau_{-x}^B$, the sequence of sets $\{\tau_{-x}^B \leq t \leq \tau_{-y}^B\}$ decreases to $\{\tau_{-x}^B = t\}$ as $y$ decreases to $x$. Thus,
	\begin{equation*}
		\lim_{y \downarrow x} P(\tau_{-x}^B \leq t \leq \tau_{-y}^B) = P(\tau_{-x}^B = t) = 0,
	\end{equation*}
	where the last equality follows from the fact that $\tau_{-x}^B$ is a continuous random variable (one can see this from \eqref{eq: reflection}).
	
	Now, let $\tau^1 = \inf \{s \geq 0 : W^0_s = \Delta\}$. Note that
	\begin{align*}
		& E\big[|f(W^0_{t - \tau_{-y}^B}) - f(W^0_{t - \tau_{-x}^B})|\mathbf{1}_{\{\tau_{-y}^B < t\}}\big] \\
		&= E\big[|f(W^0_{t - \tau_{-y}^B}) - f(W^0_{t - \tau_{-x}^B})|\mathbf{1}_{\{\tau_{-x}^B + \tau^1 = t\}}\mathbf{1}_{\{\tau_{-y}^B < t\}}\big] \\ & \quad + E\big[|f(W^0_{t - \tau_{-y}^B}) - f(W^0_{t - \tau_{-x}^B})|\mathbf{1}_{\{\tau_{-x}^B + \tau^1 \neq t\}}\mathbf{1}_{\{\tau_{-y}^B < t\}}\big]\,.
	\end{align*}
	Since $P(\{\tau_{-x}^B + \tau^1 = t) = 0$ (see Remark \ref{continuousrv} below), the first term vanishes. For the second one, note that the only discontinuity allowed of the paths for a general Brownian motion on $\R_\Delta$ occurs exactly at $\tau^1$, when it jumps from $0$ to the cemetery. Thus, on the set $\{\tau_{-x}^B + \tau^1 \neq t\}$, we have
	\begin{equation*}
		W^0_{t - \tau_{-y}^B} \longrightarrow W^0_{t - \tau_{-x}^B}  \quad \mbox{ as } y \rightarrow x
	\end{equation*}
	almost surely. Since $f$ is continuous and bounded, the Dominated Convergence Theorem yields
	\begin{equation*}
		E\big[|f(W^0_{t - \tau_{-y}^B}) - f(W^0_{t - \tau_{-x}^B})|\mathbf{1}_{\{\tau_{-x}^B + \tau^1 \neq t\}}\mathbf{1}_{\{\tau_{-y}^B < t\}}\big] \longrightarrow 0 \quad \mbox{ as } y \rightarrow x.
	\end{equation*}
	Therefore, for $x > 0$, we conclude that
	\begin{equation*}
		\big|E[f(W_t^y)] - E[f(W_t^x)]\big| \longrightarrow 0 \quad \mbox{ as } y \rightarrow x.
	\end{equation*} 
	For $x = 0$, we observe that $W_t^x$ coincides with $W_t^0$, since $\tau_0^B = 0$ almost surely. Then, for any $y \in \R$, note that
	\begin{align*}
		&\big|E[f(W_t^y)] - E[f(W_t^0)]\big| \\ &\leq E\big[|f(y + B_t) - f(W_t^0)|\mathbf{1}_{\{\tau_{-y}^B > t\}}\big] + E\big[|f(W^0_{t - \tau_{-y}^B}) - f(W_t^0)|\mathbf{1}_{\{\tau_{-y}^B \leq t\}}\big] \\
		&\leq 2\Vert f\Vert_\infty P(\tau_{-y}^B > t) + E\big[|f(W^0_{t - \tau_{-y}^B}) - f(W^0_{t })|\mathbf{1}_{\{\tau_{-y}^B \leq t\}}\big]\,.
	\end{align*}
	For the first term on the right-hand side of above, we use that $\lim_{y \to 0} \tau_{-y}^B = \tau_0^B = 0$ almost surely to obtain that
	\begin{equation*}
		\lim_{y \to 0} P(\tau_{-y}^B > t)  = 0.
	\end{equation*}
	
	For the second term, we argue as before, using the continuity of the sample paths and the function $f$, and the Dominated Convergence Theorem. This completes the proof of the continuity of $Q_tf(x)$ in $x$, and the proposition follows.
\end{proof}

\begin{observacao}\rm
	\label{continuousrv}
	The claim $P(\tau_{-x}^B + \tau^1 = t) = 0$ follows from the fact that if $X, Y$ are independent random variables and $X$ is continuous (i.e.\ its distribution function $F_X$ is a continuous function), then $X + Y$ is a continuous random variable.  	
	Indeed, for $t \in \R$ and $\delta > 0$, note that
	\begin{align*}
		P(X + Y = t) &\leq E[\mathbf{1}_{\{t - \delta < X + Y  \leq t\}}] \\
		&= E\big[E[\mathbf{1}_{\{t - \delta - Y < X \leq t - Y\}} | Y]\big] = E[g(Y)]\,,
	\end{align*}
	where $g(y) = E[\mathbf{1}_{\{t - \delta - y < X \leq t - y\}}]$. This is a consequence of the substitution principle, using that $X, Y$ are independent (see \cite[Example~5.1.5]{Durrett}). On the other hand, for all $y \in \R$ and $\eps > 0$,
	\begin{equation*}
		E[\mathbf{1}_{\{t - \delta - y < X \leq t - y\}}] = F_X(t - y) - F_X(t - \delta - y) < \eps,
	\end{equation*}
	provided that $\delta = \delta(\eps) > 0$ is chosen small enough. Here, we used that, $F_X$ is uniformly continuous, since it is continuous, $\lim_{x \to \infty} F_X(x) = 1$ and $\lim_{x \to -\infty} F_X(x)\break = 0$. Since $\eps > 0$ was arbitrary, we can conclude.
\end{observacao}

The next result characterizes the action of the generator of the general Brownian motion $W_t$ on $\bb R_\Delta$ on the functions of its domain. 
\begin{proposicao}
	\label{generator}
	For $f$ in the domain $\mf {D}(\msf{L})$ of a general Brownian motion $W_t$ on $\R_\Delta$ with boundary conditions at the origin, the infinitesimal generator equals 
	\begin{equation*}
	\msf{L}f(x) = \frac{1}{2}f''(x)\; \text{ for } x \neq 0\,, \quad   \msf{L}f(0) = \frac{1}{2}f''(0-) = \frac{1}{2}f''(0+)\,,
\end{equation*}
	 and $\msf{L}f(\Delta) = 0$. In particular,  $\mf {D}(\msf{L})$ is contained in the space of functions whose second derivative on $\R\backslash\{0\}$ extends to an element of $C(\R_{\Delta})$.
\end{proposicao}

\begin{proof}
	Let $x \neq 0$. Then $x$ is not a trap, which allows us to apply Dynkin's formula (see  \cite[Theorem 3.2.29]{Knight} for a reference), which is 
	\begin{equation}\label{teo: Dykin}
			\msf{L}f(x) \;=\; \lim_{|A| \to 0, A \ni x} \frac{E_x[f(W_{\tau_{A^\complement}})] - f(x)}{E_x[\tau_{A^\complement}]},
	\end{equation}
where $|A|$ denotes the diameter of a Borel set $A$, and as before $\tau_{A^\complement}$ denotes the first time that the process enters into the set $A^\complement$. 

Suppose that $x > 0$ and consider $A = (x - h_1, x + h_2)$ with $h_1, h_2 \searrow 0$. For $W_t$ to reach $0$ it must first exit the set $A$, provided that $h_1$ is sufficiently small. Note that until the process exits $A$, it behaves as a standard Brownian motion started at $x$ by Definition~\ref{def:GBMR}. Hence, in~\eqref{teo: Dykin} we can replace $W_t$ by a standard Brownian motion $B_t$. It is important to note that this does \emph{not} imply that they have the same generators~\cite[Chapter 7, Exercise 3.17]{RevuzYor}, since we do not know if $f$ belongs to the domain of a Brownian Motion. We do not even know at that point if $f$ is twice differentiable. Thus, let us explicitly compute the right-hand side of \eqref{teo: Dykin} for a Brownian motion:
\begin{align*}
	&E_x[f(B_{\tau_{A^\complement}})] \\ &= f(x - h_1)P_x(B_{\tau_{\{x - h_1, x + h_2\}}} = x - h_1) + f(x + h_2)P_x(B_{\tau_{\{x - h_1, x + h_2\}}} = x + h_2) \\
	&= f(x - h_1)P_0(B_{\tau_{\{- h_1, h_2\}}} = - h_1) + f(x + h_2)P_0(B_{\tau_{\{- h_1, h_2\}}} = h_2) \\
	&= \frac{f(x - h_1)h_2 + f(x + h_2)h_1}{h_2 + h_1}.
\end{align*}

Moreover, $E_x[\tau_{\{x - h_1, x + h_2\}}] = E_0[\tau_{\{- h_1, h_2\}}] = h_1h_2$, for every $x \in \R$ (see \cite[Example 3.2.30]{Knight} for the second equality).
Thus, we obtain
	\begin{align}
		\label{DynkinformulaWt}
		\msf{L}f(x) &= \lim_{h_1 \searrow 0, h_2 \searrow 0} \left(\frac{f(x - h_1)h_2 + f(x + h_2)h_1}{h_2 + h_1} - f(x)\right)(h_1h_2)^{-1} \nonumber \\
		&= \lim_{h_1 \searrow 0, h_2 \searrow 0}  \left(\frac{f(x - h_1) - f(x)}{h_1} + \frac{f(x + h_2) - f(x)}{h_2}\right)(h_1 + h_2)^{-1}.
	\end{align}
	
	Since $h_1$ and $h_2$ go to $0$ independently, we can take first the limit on $h_2$ and we get
	\begin{equation}
		\label{derivativeexists}
		\msf{L}f(x) = \lim_{h_1 \searrow 0}  \left(\frac{f(x - h_1) - f(x)}{h_1} + \lim_{h_2 \searrow 0} \frac{f(x + h_2) - f(x)}{h_2}\right)(h_1)^{-1},
	\end{equation}
	so $\lim_{h_2 \searrow 0} \frac{f(x + h_2) - f(x)}{h_2}$ exists, which is precisely $f'_+(x)$ the right derivative of $f$ at $x$. Similarly, the limit $\lim_{h_1 \searrow 0} \frac{f(x - h_1) - f(x)}{h_1}$ exists and it is equal to $- f'_-(x)$, where $f'_-(x)$ denotes the left derivative of $f$. Moreover, for the limit in \eqref{derivativeexists} to be finite, the numerator must necessarily approach zero, implying that $f'_-(x) = f'_+(x)$. This ensures the existence of $f'(x)$ for $x > 0$. A completely analogous argument shows that $f'(x)$ also exists for $x < 0$.
	
	Now, we compute the second derivative. Fix $0 < \eps < x$, we apply \eqref{DynkinformulaWt} to $\eps < y < x$ with $h_1 = h_2 = h_n = \frac{x - \eps}{n}$, for each $n \in \N$, so 
	\begin{equation*}
		\msf{L}f(y) =	\lim_{n \to \infty} \frac{1}{2h_n}\left(\frac{f(y - h_n) - f(y)}{h_n} + \frac{f(y + h_n) - f(y)}{h_n}\right).
	\end{equation*}
	Since $Lf$ is continuous,  it is bounded on $[\eps, x]$, so by Dominated Convergence Theorem we may exchange the limit and the integral below: 
	\begin{equation*}
		\int_{\eps}^x \msf{L}f(y) \, dy = \lim_{n \to \infty} \int_{\eps}^x \frac{1}{2h_n}\left(\frac{f(y - h_n) - f(y)}{h_n} + \frac{f(y + h_n) - f(y)}{h_n}\right) \, dy.
	\end{equation*} 
	Observe that 
	\begin{align*}
		& \int_{\eps}^x \frac{1}{2h_n}\left(\frac{f(y - h_n) - f(y)}{h_n} + \frac{f(y + h_n) - f(y)}{h_n}\right) \, dy \\
		&= \frac{1}{2h_n^2} \left(\int_{\eps - h_n}^{x - h_n} f(y) \, dy - \int_{\eps}^x f(y) \, dy\right) + \frac{1}{2h_n^2}\left(\int_{\eps + h_n}^{x + h_n} f(y) \, dy - \int_{\eps}^x f(y) \, dy \right) \\
		&= \frac{1}{2h_n^2} \left(\int_{\eps - h_n}^{\eps}\!\!\! f(y) \, dy - \int_{x - h_n}^x\!\!\! f(y) \, dy \right) + \frac{1}{2h_n^2} \left(\int_{x}^{x + h_n}\!\!\! f(y) \, dy - \int_{\eps}^{\eps + h_n}\!\!\! f(y) \, dy \right).
	\end{align*}
	
	Define $g(x) = \int_{0}^x f(y) \, dy$ for $x > 0$. Since $f$ is differentiable on $(0, \infty)$, we have that $g$ is twice differentiable and $g'' = f'$ for $x > 0$. We can rewrite the above expression in terms of $g$ as follows:
	\begin{align*}
		& \frac{1}{2h_n^2} \left(\int_{x}^{x + h_n} f(y) \, dy - \int_{x - h_n}^x f(y) \, dy \right) - \frac{1}{2h_n^2} \left(\int_{\eps}^{\eps + h_n} f(y) \, dy - \int_{\eps - h_n}^{\eps} f(y) \, dy \right) \\
		&= \frac{1}{2h_n}\left(\frac{g(x + h_n) - g(x)}{h_n} + \frac{g(x - h_n) - g(x)}{h_n}\right) \\ & \quad \quad - \frac{1}{2h_n}\left(\frac{g(\eps + h_n) - g(\eps)}{h_n} + \frac{g(\eps - h_n) - g(\eps)}{h_n}\right),
	\end{align*}
	and an application of Taylor's theorem yields
	\begin{align*}
		\lim_{n \to \infty} \frac{1}{2h_n}\left(\frac{g(x + h_n) - g(x)}{h_n} + \frac{g(x - h_n) - g(x)}{h_n}\right) = \frac{1}{2}g''(x) = \frac{1}{2}f'(x),
	\end{align*}
	and
	\begin{align*}
		\lim_{n \to \infty} \frac{1}{2h_n}\left(\frac{g(\eps + h_n) - g(\eps)}{h_n} + \frac{g(\eps - h_n) - g(\eps)}{h_n}\right) = \frac{1}{2}f'(\eps).
	\end{align*}
	
	Thus, collecting all these expressions, we have
	\begin{equation*}
		\int_{\eps}^x \msf{L}f(y) \, dy = \frac{1}{2}(f'(x) - f'(\eps)),
	\end{equation*} 
	and the Fundamental Theorem of Calculus implies that $f'$ is differentiable and
	\begin{equation*}
		\msf{L}f(x) = \frac{1}{2}f''(x),
	\end{equation*}
	for $x > \eps > 0$. Since $\eps \in (0, x)$ was arbitrary, this identity extends to all $x > 0$. The proof for $x < 0$ is analogous.	
The generator at $\Delta$ is trivially derived by noticing that $\Delta$ is an absorbing state:
	\begin{equation*}
		\msf{L}f(\Delta) = \lim_{t \to 0^+} \frac{E_{\Delta}[f(W_t)] - f(\Delta)}{t} = 0.
	\end{equation*}
Therefore, for $f \in \mf {D}(\msf{L})$, we have
	\begin{equation*}
		\msf{L}f(x) = \frac12f''(x)\; \mbox{ for } x \neq 0 \quad \mbox{ and } \quad \msf{L}f(\Delta) = 0.
	\end{equation*}
Since $\msf{L}f$ is continuous, there exists a natural extension of $f''$ to $0$ 
	\begin{align*}
		2\msf{L}f(0) = \lim_{x \to 0^+} f''(x) = \lim_{x \to 0^-} f''(x)\,.
	\end{align*}
This extension  belongs to $C(\R_{\Delta})$, which completes the proof.
\end{proof}

\begin{observacao}\rm
	\label{rem:dif}
Note that for $f \in \mf {D}(\msf{L})$, one has that
	\begin{equation*}\label{eq:firstderivative}
		f'(0+) = f'_+(0)\quad \text{ and }\quad  f'(0-) = f'_-(0),
	\end{equation*}
	where $f'_+$ ($f'_-$) denotes the derivative from the right (left).
	Indeed, let $f \in \mf {D}(\msf{L})$. We have seen in Proposition~\ref{generator} that $f$ is twice differentiable on $\R\backslash \{0\}$, and the second derivative $f''$ extends to an element of $C(\R_{\Delta})$, in particular $f''$ is bounded. Thus, for any sequence $c_n \downarrow 0$, we have
	\begin{equation*}
		|f'(c_n) - f'(c_m)| = \left|\int_{c_m}^{c_n} f''(x) \, dx\right| \leq \Vert f''\Vert_\infty \cdot |c_n - c_m| \longrightarrow 0,
	\end{equation*}
	for $n, m$ large enough. This proves that the limit $\lim_{c \to 0^+} f'(c)$ exists.
	
	Now, for any $h > 0$, we have that $f$ is continuous on $[0, h]$ and differentiable on $(0, h)$, so using the Mean Value Theorem, we can find a $c = c(h) \in (0, h)$ such that
	\begin{equation*}
		f'(c) = \frac{f(h) - f(0)}{h}.
	\end{equation*}
	In particular, since $c(h) \rightarrow 0$ as $h \rightarrow 0$, it follows that
	\begin{equation}
		\label{meanvaluethm}
		f'_+(0) = \lim_{h \to 0^+} \frac{f(h) - f(0)}{h} = \lim_{c \to 0^+} f'(c) = f'(0+),
	\end{equation}
	as desired. The other limit can be computed similarly.
\end{observacao}

Now, in order to fully characterize the general Brownian motion on $\R_\Delta$ with boundary conditions at the origin, we must determine the domain of its infinitesimal generator. Naturally, we expect that the domain of our general Brownian motion has certain restrictions at zero. The next lemma confirms our expectation by showing that the domain $\mf {D}(\msf{L})$ is determined by a condition in the neighborhood of $x = 0$. Moreover, we will see later that this boundary condition depends on the behavior of the process when it hits zero.

\begin{lema}
	\label{neighb0}
	Let $f_1 \in \mf {D}(\msf{L})$. If  $f_2 \in C(\R_{\Delta})$ is such that $f_2''$ exists in $\R \backslash  \{0\}$, and $f_2''$ admits an extension to an element of $C(\R_{\Delta})$ and if moreover $f_1 = f_2$ in $[-\eps, \eps]$ for some $\eps > 0$, then $f_2 \in \mf {D}(\msf{L})$.
\end{lema}

\begin{proof}
	We need to show that the limit 
	\begin{equation*}
	\lim_{t \to 0^+} \frac{E_x[f_2(W_t)] - f_2(x)}{t}
	\end{equation*}
	exists and that it is uniform in $x$. We split the analysis into two cases: $|x| > \eps/2$ and $|x| < \eps/2$.
	
	Let $|x| > \eps/2$. We may assume that $x>0$. Since the process $W_t$ is equal in distribution to $B_t$ until the hitting time of zero, we compare the above limit to the one where $W$ is replaced by $B^0$, an absorbed Brownian motion at $0$ starting from $x$, then
	\begin{align*}
		& \left|\frac{E_x\big[f_2(W_t)\big] - f_2(x)}{t} - \frac{E_x\big[f_2(B^0_t)\big] - f_2(x)}{t}\right|\\
		&= \left|\frac{E_x\big[f_2(B^0_t)\mathbf{1}_{\{\tau_0 > t\}}\big] + E_x\big[f_2(W_t)\mathbf{1}_{\{\tau_0 < t\}}] - f_2(x)}{t} - \frac{E_x[f_2(B^0_t)\big] - f_2(x)}{t}\right| \\
		&= \left|\frac{E_x\big[f_2(W_t)\mathbf{1}_{\{\tau_0 < t\}}\big] - E_x\big[f_2(B^0_t)\mathbf{1}_{\{\tau_0 < t\}}\big]}{t}\right| \\
		&\leq 2\Vert f_2\Vert_\infty \frac{P_x(\tau_0 < t)}{t}.
	\end{align*} 
		By the Reflection Principle,  we have that
					\begin{align}
			\label{estimativa hitting time 0}
			P_x(\tau_0 < t) & = P_0(\tau_x < t) = P_0(|B_t| \geq x) \leq \frac{E_0[B_t^{2n}]}{x^{2n}}
			\lesssim \frac{t^n}{\varepsilon^{2n}}\,.
		\end{align}
	for  all $n \in \N$ and $|x| > \eps$. Therefore, we conclude that the  limits 
	\begin{equation*}
		\lim_{t \to 0^+} \frac{E_x[f_2(W_t)] - f_2(x)}{t} \quad \text{ and }\quad  \lim_{t \to 0^+} \frac{E_x[f_2(B^0_t)] - f_2(x)}{t}
	\end{equation*}
	must  be equal for $|x| > \eps/2$ whenever they exist, and that they are uniform in $x$ (provided that $\varepsilon$ is fixed).
	Note that $f_2 \in C(\R_\Delta)$, is twice differentiable in $\R \backslash  \{0\}$ and $f_2''(0+)$ exists. This however, does not imply that $f_2$ is in the domain of the absorbed Brownian motion since this would require that $f_2''(0+)=0$. However, a small computation using~\eqref{estimativa hitting time 0} shows that it still holds that
	\begin{equation*}
		\lim_{t \to 0^+} \frac{E_x[f_2(B^0_t)] - f_2(x)}{t} \;=\;\frac{1}{2}f_2''(x)\,.
	\end{equation*} 
	It remains to prove that the same limit extends uniformly near to $x = 0$. Since $f_1 = f_2$ on $[-\eps, \eps]$, we have for $-\eps < x < \eps$ that
	\begin{align*}
		&\left|\frac{E_x[f_2(W_t)] - f_2(x)}{t} - \frac{E_x[f_1(W_t)] - f_1(x)}{t}\right|\\
		&	= \left|\frac{E_x[f_2(W_t)\mathbf{1}_{\{|W_t| > \eps\}}] - E_x[f_1(W_t)\mathbf{1}_{\{|W_t| > \eps\}}}{t}\right|.
	\end{align*}
	
	Set $T_{\eps} = \tau_{\{-\eps, \eps\}}$. If $|W_t| > \eps$, then $W$ has already hit $\eps$ or $-\eps$ by time $t$, and $T_{\eps} < t$. Observing also that $|f_2(x) - f_1(x)| \leq \max_{|x| \geq \eps} |f_2(x) - f_1(x)|$, which is finite, we get
	\begin{equation*}
		\left|\frac{E_x[f_2(W_t)] - f_2(x)}{t} - \frac{E_x[f_1(W_t)] - f_1(x)}{t}\right| \leq \max_{|x| \geq \eps} |f_2(x) - f_1(x)|\frac{P_x(T_{\eps} < t)}{t}.
	\end{equation*}
	
	We now estimate this last probability and aim to find a uniform upper bound. Considering $-\eps/2 \leq x \leq \eps/2$, we have $P_x(T_{\eps/2} \leq T_{\eps}) = 1$ by continuity of the paths (note that this also includes the case when $T_{\eps} = \infty$ when the process is killed before hitting $-\eps$ or $\eps$). Now, for $s\geq 0$ consider the shift operators $\theta_s$ defined for any c\`adl\`ag function $\omega$ via $((\theta_s\omega)_t)_{t\geq 0}= (\omega_{t+s})_{t\geq 0}$. Thus, if $W$ reaches $-\eps$ or $\eps$ before time $t$, then the shifted process $W \circ \theta_{T_{\eps/2}}$ also reaches $-\eps$ or $\eps$ before time $t$. Hence,
	\begin{equation*}
		\mathbf{1}_{\{T_{\eps} < t\}} \leq \mathbf{1}_{\{T_{\eps} < t\}} \circ \theta_{T_{\eps / 2}}. 
	\end{equation*}
	An application of the strong Markov property yields
	\begin{align*}
		P_x(T_{\eps} < t) &\leq E_x[\mathbf{1}_{\{T_{\eps} < t\}} \circ \theta_{T_{\eps / 2}}] \\
		&= E_x[E_{W_{T_{\eps / 2}}}[\mathbf{1}_{\{T_{\eps} < t\}}]] \\
		&= E_x[\mathbf{1}_{\{W_{T_{\eps/2} = -\eps/2}\}}E_{-\eps/2}[\mathbf{1}_{\{T_{\eps} < t\}}]] + E_x[\mathbf{1}_{\{W_{T_{\eps/2} = \eps/2}\}}E_{\eps/2}[\mathbf{1}_{\{T_{\eps} < t\}}]] \\
		&= P_x(W_{T_{\eps/2}} = -\eps/2)P_{-\eps/2}(T_{\eps} < t) + P_x(W_{T_{\eps/2}} = \eps/2)P_{\eps/2}(T_{\eps} < t) \\
		&\leq P_{-\eps/2}(T_{\eps} < t) + P_{\eps/2}(T_{\eps} < t).
	\end{align*}
	
	Now, we compare this to Brownian motion hitting times. Let $T_{\{0, \eps\}} = \inf \{s \geq 0 : W_s \in \{-\eps, 0, \eps\}\}$ and $T_{\{0, \eps\}}^B = \inf \{s \geq 0 : B_s \in \{-\eps, 0, \eps\}\}$, where $(B_s)_{s \geq 0}$ is a standard Brownian motion. Observe that
	\begin{align*}
		P_{\eps/2}(T_{\eps} < t) &\leq P_{\eps/2}(T_{\{0, \eps\}} < t) = P_{\eps/2}(T^B_{\{0, \eps\}} < t)\,,
	\end{align*}
	where we have used that $W_t$ is equal in distribution to a Brownian motion in time interval $[0, T_{\{0, \eps\}}]$. Similarly, $P_{-\eps/2}(T_{\eps} < t) \leq P_{-\eps/2}(T^B_{\{0, \eps\}} < t)$. Also, as we did in~\eqref{estimativa hitting time 0}, we have
	\begin{equation}\label{eq:smallt}
		P_{-\eps/2}(T^B_{\{0, \eps\}} < t) = P_{\eps / 2}(T_{\{0, \eps\}}^B < t) = o(t^n) \quad \mbox{ as } t \rightarrow 0^{+}\,
	\end{equation}
	for all $n > 0$.	
	Putting all  together, we get
	\begin{align*}
		& \left|\frac{E_x[f_2(W_t)] - f_2(x)}{t} - \frac{E_x[f_1(W_t)] - f_1(x)}{t}\right| \\ &\leq\; \max_{x \geq \eps} |f_2(x) - f_1(x)|\frac{P_{-\eps/2}(T^B_{\{0, \eps\}} < t) + P_{\eps/2}(T^B_{\{0, \eps\}} < t)}{t}
	\end{align*} 
	which goes to zero independently of $x$ as $t \rightarrow 0^{+}$. Therefore, both expressions have the same limit. Recalling that $f_1 \in \mf {D}(\msf{L})$, we obtain that
	\begin{equation*}
		\lim_{t \to 0^+} \frac{E_x[f_2(W_t)] - f_2(x)}{t} = \lim_{t \to 0^+} \frac{E_x[f_1(W_t)] - f_1(x)}{t}
	\end{equation*}
	uniformly in $|x| \leq \eps/2$. This completes the proof.
\end{proof}

The next result will bring us one step closer in characterising the domain of the general Brownian motion on $\R_\Delta$. 

\begin{lema}
	\label{quaseFeller}
	There exist constants $c_1, c^-_2, c^+_2, c_3 \geq 0$ and measures $\nu^-(dx)$ and $\nu^+(dx)$ on $(-\infty, 0)$ and  $(0, \infty)$, respectively, satisfying
	\begin{equation}
		\label{constants normalization}
		c_1 + c^-_2 + c^+_2 + c_3 + \int_{-\infty}^{0} (1 \wedge - x) \, \nu^-(dx) + \int_{0}^{\infty} (1 \wedge x) \, \nu^+(dx) \;=\; 1\,,
	\end{equation}
	so that $f$ belongs to $\mf {D}(\msf{L})$ if and only if: $f$ belongs to $C(\R_{\Delta})$, is twice differentiable in $\R \backslash  \{0\}$ with second derivative extending to an element of $C(\R_{\Delta})$, and the equation
	\begin{align}
		\label{firstdomain}
		& \nonumber c_1f(0) + c^-_2f'(0-) - c^+_2f'(0+) + \frac{c_3}{2}f''(0+) \\ &= \int_{(-\infty, 0)} (f(x) - f(0)) \, \nu^-(dx) + \int_{(0, \infty)} (
		f(x) - f(0)) \, \nu^+(dx)
	\end{align}
	holds.
\end{lema}

To prove the previous lemma we need the next result concerning weak convergence of measures.
\begin{proposicao}
	\label{weakconvergencemeasures}
	Let $(\mu_{\eps})_{\eps > 0}$ be a family of measures on $(0, \infty)$ such that\linebreak $\mu_{\eps}((0, \infty)) \leq 1$ for all $\eps > 0$, and let $C[0, \infty]$ be the space of continuous functions on $[0, \infty]$. Then, there exist a subsequence $\eps_n \searrow 0$ and a measure $\mu$ on $[0, \infty]$ such that 
	\begin{equation}\label{eq:f-vague}
		\lim_{n \to \infty} \int_{(0, \infty)} f(x) \, \mu_{\eps_n}(dx) = \int_{[0, \infty]} f(x) \, \mu(dx)
	\end{equation}
	for every $f \in C[0, \infty]$. 
\end{proposicao}

\begin{observacao}\rm
	The space $C[0, \infty]$ corresponds to the space  of continuous functions $f:\bb R_{\Delta,+}\to \bb R$ without the convention that $f(\Delta)=0$. After all,  $\bb R_{\Delta,+} = [0,\infty)\cup \{\Delta\}$ is the one-point compactification of $[0,\infty)$, where $\Delta$ is the ``point at infinity''. 
	%So if $f:\bb R^\Delta_+\to \bb R$ is a continuous function, then there exists the limit towards $\Delta$.
\end{observacao}

\begin{observacao}\rm
	Since $\bb R_{\Delta,+}$ is  a compact space, the reader may wonder if the Proposition~\ref{weakconvergencemeasures} could be a direct consequence of Banach-Alaoglu's Theorem from Functional Analysis (together with  Riesz Representation Theorem). That is not the case because   the domain of integration in the left hand side of \eqref{eq:f-vague} is $(0,\infty)$ while the domain on the right hand side is $[0,\infty]$. 
\end{observacao}

\begin{proof}[Proof of Proposition~\ref{weakconvergencemeasures}]
	The space $C[0, \infty]$ has a countable dense subset in the supremum norm, which we denote by $\{f_k\}_{k = 1}^{\infty}$. We will show first that there exists a subsequence $\eps_n \searrow 0$ such that the limit
	\begin{equation*}
		\lim_{n \to \infty} \int_{(0, \infty)} f_k(x) \, \mu_{\eps_n}(dx)
	\end{equation*}
	exists, for all $k \in \N$.  We  will proceed via a  diagonal argument. Letting $k = 1$, we have
	\begin{equation*}
		\int_{(0, \infty)} f_1(x) \, \mu_{\eps}(dx) \leq \Vert f_1\Vert_\infty\mu_{\eps}((0, \infty)) \leq \|f_1\|,
	\end{equation*}
	for all $\eps > 0$. Thus, we can find a sequence $\eps^1_n \searrow 0$ such that $\int_{(0, \infty)} f_1(x) \, \mu_{\eps^1_n}(dx)$ converges as $n \rightarrow \infty$. Also, for $k = 2$, we have 
	\begin{equation*}
		\int_{(0, \infty)} f_2(x) \, \mu_{\eps^1_n}(dx) \leq \Vert f_2\Vert_\infty,
	\end{equation*}
	for all $\eps^1_n > 0$, then we can extract a subsequence $(\eps^2_n)$ of $(\eps^1_n)$ such that\linebreak $\int_{0}^{\infty} f_2(x) \, \mu_{\eps^2_n}(dx)$ converges as $n \rightarrow \infty$ with $\eps^2_n \searrow 0$. Applying this successively, we will have found for each $k$ a sequence $(\eps^k_n) \subseteq (\eps^{k-1}_n) \searrow 0$ such that $\int_{0}^{\infty} f_k(x) \, \mu_{\eps^k_n}(dx)$ converges as $n \rightarrow \infty$. 
		Thus, fixing $k$ and taking the diagonal sequence $\eps_n = \eps^n_n$, we obtain
	\begin{equation*}
		\left|\int_{(0, \infty)} f_k(x) \, \mu_{\eps^n_n}(dx) - \int_{(0, \infty)} f_k(x) \, \mu_{\eps^k_n}(dx) \right| \longrightarrow 0,
	\end{equation*}
	as $n \rightarrow \infty$, since for $n$ sufficiently large, we have $(\eps^n_n) \subseteq (\eps^k_n)$ and the second integral converges. This completes the diagonal argument.
	
	Next, we can extend this limit for all $f \in C[0, \infty]$ using the density of $\{f_k\}_{k = 1}^{\infty}$. In fact, for every function $f \in C[0, \infty]$, there exists a subsequence $(f_{k_n})_n$ of $\{f_k\}_{k = 1}^{\infty}$ such that $f_{k_n} \longrightarrow f$ as $n \rightarrow \infty$ in the supremum norm. Then, for all $\eps > 0$, we can choose $n$ large enough such that $\|f_{k_n} - f \|_{\infty} < \eps$. Hence,
	\begin{equation*}
		\left|\int_{(0, \infty)} f(x) \, \mu_{\eps_n}(dx) - \int_{(0, \infty)} f_{k_n}(x) \, \mu_{\eps_n}(dx) \right| \leq \|f_{k_n} - f \|_{\infty} < \eps\,.
	\end{equation*}
	Thus the limit $\lim_{n \to \infty} \int_{0}^{\infty} f(x) \, \mu_{\eps_n}(dx)$ exists for every $f \in C[0, \infty]$, for the same sequence $\eps_n \searrow 0$.
	
	Now, we may define the operator $C[0, \infty] \ni f \mapsto \lim_{n \to \infty} \int_{0}^{\infty} f(x) \, \mu_{\eps_n}(dx)$, which is positive and linear. Also, since $C[0, \infty]$ consists of bounded continuous functions on a compact space, Riesz Representation Theorem ensures that there exists a measure $\mu(dx)$ on $[0, \infty]$ such that
	\begin{equation*}
		\lim_{n \to \infty} \int_{(0, \infty)} f(x) \, \mu_{\eps_n}(dx) = \int_{[0, \infty]} f(x) \, \mu(dx),
	\end{equation*}
	for all $f \in C[0, \infty]$, which completes the proof.
\end{proof}

\begin{observacao}
	\label{weakconv2}
	Note that the same result holds when considering a family $(\mu_{\eps})_{\eps > 0}$ of measures on $(-\infty, 0)$ with $\mu_{\eps}((-\infty, 0)) \leq 1$, for all $\eps > 0$, and the space $C[-\infty, 0]$. 
\end{observacao}

We now come back to the proof of Lemma \ref{quaseFeller}.
\begin{proof}[Proof of Lemma \ref{quaseFeller}]
	First, due to Proposition \ref{prop:jump time}, $T = \inf \{t > 0 : W_t \neq 0\}$ is an exponential random variable with parameter $\lambda \in [0, \infty]$ under $P_0$. We will split the proof into three cases: $\lambda = 0, 0 < \lambda < \infty$, and $\lambda = \infty$.

	\textbf{Case}  $\lambda = 0$. 	Here, $T = \infty$ a.s. and $0$ is a trap, so the process $W_t$ must coincide with the absorbed Brownian motion at $0$. Also, since $0$ is absorbing, we have
	\begin{equation*}
		\msf{L}f(0) = \lim_{t \to 0^+} \frac{E_0[f(W_t)] - f(0)}{t} = 0,
	\end{equation*}
	for any $f \in \mf {D}(\msf{L})$. Thus, by Proposition~\ref{generator} we have that $f''(0-) = f''(0+) = 0$, and choosing $c_1 = c^-_2 = c^+_2 = 0$,  $c_3 = 1$, and $\nu^- = \nu^+ = 0$, the lemma holds trivially. 
	
	\textbf{Case} $0 < \lambda < \infty$. 	We will show that $\frac{1}{2}f''(0+) = -\lambda f(0)$ is the boundary condition for $f \in \mf {D}(\msf{L})$. Again by Proposition \ref{prop:jump time}, the process leaves $0$ jumping to the cemetery. Then, for $f \in C(\R_{\Delta})$, we get
	\begin{align*}
		E_0[f(W_t)] &= E_0[f(W_t)\mathbf{1}_{\{T < t\}}] + E_0[f(W_t)\mathbf{1}_{\{T > t\}}] \\
		&= E_0[f(\Delta)\mathbf{1}_{\{T < t\}}] + E_0[f(0)\mathbf{1}_{\{T > t\}}] \\
		&= f(0)P_0(T > t) \\
		&= f(0)\exp(-\lambda t).
	\end{align*}
	Thus, for $f \in \mf {D}(\msf{L})$,
	\begin{equation*}
		Lf(0) = \lim_{t \to 0^+} \frac{E_0[f(W_t)] - f(0)}{t} =  f(0)\lim_{t \to 0^+} \frac{\exp(-\lambda t) - 1}{t} = -\lambda f(0),
	\end{equation*}
	which implies that $\frac{1}{2}f''(0+) = -\lambda f(0)$, as desired. With this boundary condition, it is enough to choose $c_1 = \frac{\lambda}{1 + \lambda}, c_3 = \frac{1}{1 + \lambda}$, $c^-_2 = c^+_2 = 0$, and $\nu^- = \nu^+ = 0$ to obtain the result.\bigskip

	\textbf{Case}  $\lambda = \infty$. 	In this situation, $T = 0$ a.s., and $0$ is not a trap, so we may compute the generator at $x = 0$ via Dynkin's formula \eqref{teo: Dykin}. Let $A = (-\eps, \eps)$ with $\eps \searrow 0$. Starting from $0$, $W$ exits $A$ at $-\eps, \eps$ or at $\Delta$. Setting $T_{\eps} = \tau_{-\eps} \wedge \tau_{\eps} \wedge \tau_{\Delta}$, Dynkin's formula \eqref{teo: Dykin} becomes		
	\begin{equation}
		\label{eq: gen 0}
		\msf{L}f(0) = \lim_{\eps \searrow 0} \frac{E_0[f(W_{T_{\eps}})] - f(0)}{E_0[T_{\eps}]},
	\end{equation}
	for $f \in \mf {D}(\msf{L})$. Note that
	\begin{align*}
		\frac{E_0[f(W_{T_{\eps}})]}{E_0[T_{\eps}]} &= \int \frac{1}{E_0[T_{\eps}]}f(W_{T_{\eps}}) \, dP_0 \\
		&= \int_{\R_{\Delta}} f(x) \, \frac{1}{E_0[T_{\eps}]}P_0(W_{T_{\eps}} \in dx) = \int_{\R} f(x) \, \frac{1}{E_0[T_{\eps}]}P_0(W_{T_{\eps}} \in dx),
	\end{align*}
	where in the last equality we used that $f$ vanishes at $\Delta$. 	Thus, considering the \emph{exit measure} 
	\begin{equation}
		\label{eq:exitmeas}
		\nu_{\eps}(dx) = \frac{1}{E_0[T_{\eps}]}P_0(W_{T_{\eps}} \in dx)\, \quad \text{on }\R_\Delta\,,
	\end{equation} 
	we can write 
	\begin{equation}
		\label{eqgen}
		\msf{L}f(0) = \lim_{\eps \searrow 0} \,\bigg\{ \bigg[\int_{\R} (f(x) - f(0)) \, \nu_{\eps}(dx)\bigg] -f(0)\nu_{\eps}(\Delta)\bigg\}.
	\end{equation}
	In order to obtain normalized constants as in \eqref{constants normalization}, let
	\begin{equation*}
		K_{\eps} = 1 + \nu_{\eps}(\Delta) + \int_{-\infty}^{0} (1 \wedge -x) \, \nu_{\eps}(dx) + \int_{0}^{\infty} (1 \wedge x) \, \nu_{\eps}(dx).
	\end{equation*}
	To continue, consider the measures
	\begin{align*}
		\mu^-_{\eps}(A) &= \int_A (1 \wedge - x) \,  \frac{\nu_{\eps}(dx)}{K_{\eps}} \text{ on } (-\infty, 0)\qquad \text{ and} \\
		\mu^+_{\eps}(A) &= \int_A (1 \wedge x) \, \frac{\nu_{\eps}(dx)}{K_{\eps}} \text{ on } (0, \infty).
	\end{align*}
	Observe that $\mu^+_{\eps}(dx) \leq 1$ for all $\eps > 0$, hence Proposition \ref{weakconvergencemeasures} assures that there is a measure $\mu^+(dx)$ on $[0, \infty]$ and a subsequence $\eps_n \searrow 0$ such that
	\begin{equation}
		\label{eqweakconv}
		\lim_{n \to \infty} \int_{(0, \infty)} f(x) \, \mu^+_{\eps_n}(dx) = \int_{[0, \infty]} f(x) \, \mu^+(dx),
	\end{equation}
	for all $f \in C[0, \infty]$. In view of Remark \ref{weakconv2}, applying the same argument to $\mu^-_{\eps_n}$ and passing to a further subsequence if necessary, we get a measure $\mu^-(dx)$ on $[-\infty, 0]$ such that
	\begin{equation*}
		\lim_{n \to \infty} \int_{(-\infty, 0)} f(x) \, \mu^-_{\eps_n}(dx) = \int_{[-\infty, 0]} f(x) \, \mu^-(dx),
	\end{equation*} 
	for every $f \in C[-\infty, 0]$.
	
	Let $f \in \mf {D}(\msf{L})$. Since $f \in C(\R_{\Delta})$, we have that the function $\frac{f(x) - f(0)}{1 \wedge x}$ is continuous on $(0, \infty)$ and $\lim_{x \to \infty} \frac{f(x) - f(0)}{1 \wedge x} = -f(0) < \infty$. Also, we know that $f$ is differentiable, so $\lim_{x \to 0^+} \frac{f(x) - f(0)}{1 \wedge x} = \lim_{x \to 0^+} \frac{f(x) - f(0)}{x} = f'(0+) < \infty$. Thus, we can extend $\frac{f(x) - f(0)}{1 \wedge x}$ continuously to $[0, \infty]$.
	
	Using Equation \eqref{eqweakconv} with the function $\frac{f(x) - f(0)}{1 \wedge x} \in C[0, \infty]$ and observing that $1 \wedge x$ is the Radon-Nikodym derivative of $\mu^+_{\eps_n}(dx)$ with respect to $\frac{1}{K_{\eps_n}}\nu_{\eps_n}(dx)$, we get
	\begin{align*}
		\int_{[0, \infty]} \frac{f(x) - f(0)}{1 \wedge x} \, \mu^+(dx) &= \lim_{n \to \infty} \int_{(0, \infty)} \frac{f(x) - f(0)}{1 \wedge x} \, \mu^+_{\eps_n}(dx) \\ 
		&= \lim_{n \to \infty} \int_{(0, \infty)} \frac{f(x) - f(0)}{1 \wedge x} (1 \wedge x) \, \frac{1}{K_{\eps_n}}\nu_{\eps_n}(dx) \\
		&= \lim_{n \to \infty} \int_{(0, \infty)} \big\{f(x) - f(0)\big\} \, \frac{1}{K_{\eps_n}}\nu_{\eps_n}(dx).
	\end{align*}
	Similarly, $\frac{f(x) - f(0)}{1 \wedge -x}$ is continuous on $(-\infty, 0)$, $\lim_{x \to -\infty} \frac{f(x) - f(0)}{1 \wedge -x} = -f(0) < \infty$, and $\lim_{x \to 0^-} \frac{f(x) - f(0)}{1 \wedge -x} = \lim_{x \to 0^-} \frac{f(x) - f(0)}{-x} = -f'(0-) < \infty$. Thus, 
	\begin{equation*}
		\int_{[-\infty, 0]} \frac{f(x) - f(0)}{1 \wedge -x} \, \mu^-(dx) =  \lim_{n \to \infty} \int_{(-\infty, 0)} \big\{f(x) - f(0)\big\} \, \frac{1}{K_{\eps_n}}\nu_{\eps_n}(dx).
	\end{equation*}
Note that since $0 \leq \frac{\nu_{\eps_n}(\Delta)}{K_{\eps_n}} \leq 1$ and $0 < \frac{1}{K_{\eps_n}} \leq 1$, for all $n \in \N$, we may assume
\begin{equation*}
	\lim_{n \to \infty} \frac{\nu_{\eps_n}(\Delta)}{K_{\eps_n}} = p_1 \quad \text{ and }\quad \lim_{n \to \infty} \frac{1}{K_{\eps_n}} = p_2\,,
\end{equation*}
passing to a subsequence if necessary. Therefore, from Equation~\eqref{eqgen} we obtain
	\begin{equation*}
		p_1 f(0) + p_2\msf{L}f(0) - \int_{[-\infty, 0]} \frac{f(x) - f(0)}{1 \wedge -x} \, \mu^-(dx) - \int_{[0, \infty]} \frac{f(x) - f(0)}{1 \wedge x} \, \mu^+(dx) = 0.
	\end{equation*}
	Considering the values of $\frac{f(x) - f(0)}{1 \wedge -x}$ and $\frac{f(x) - f(0)}{1 \wedge x}$ at $0, \pm \infty$, we get
	\begin{align*}
		& (p_1 + \mu^-(-\infty) + \mu^+(\infty))f(0) + \mu^-(0)f'(0-) - \mu^+(0)f'(0+) + p_2\msf{L}f(0) \\ &= \int_{(-\infty, 0)} \frac{f(x) - f(0)}{1 \wedge -x} \, \mu^-(dx) + \int_{(0, \infty)} \frac{f(x) - f(0)}{1 \wedge x} \, \mu^+(dx).
	\end{align*}
	
	Defining the measures $\nu^-(A) = \int_{A} \frac{1}{1 \wedge -x} \mu^-(dx)$ on $(-\infty, 0)$,\linebreak $\nu^+(A) = \int_{A} \frac{1}{1 \wedge x} \mu^+(dx)$ on $(0, \infty)$, and setting $c_1 = p_1 + \mu^-(-\infty) + \mu^+(\infty)$, $c^-_2 = \mu^-(0)$, $c^+_2 = \mu^+(0)$, $c_3 = p_2$ yields
	\begin{align*}
		& c_1f(0) + c^-_2f'(0-) - c^+_2f'(0+) + \frac{c_3}{2}f''(0+) \\&= \int_{(-\infty, 0)} (f(x) - f(0)) \, \nu^-(dx) + \int_{(0, \infty)} (
		f(x) - f(0)) \, \nu^+(dx),
	\end{align*}
	which is \eqref{firstdomain}. It remains to verify Equation~\eqref{constants normalization}.  By integrating against the function that is constant equal to one, we have that $\mu^-([- \infty, 0]) =\break \lim_{n \to \infty} \mu^-_{\eps_n}((-\infty, 0))$ and $\mu^+([0, \infty]) = \lim_{n \to \infty} \mu^+_{\eps_n}((0, \infty))$. Thus,
	\begin{align*}
		& p_1 + p_2 + \mu^-([-\infty, 0]) + \mu^+([0, \infty]) \\ &= \lim_{n \to \infty} \Big(\frac{\nu_{\eps_n}(\Delta)}{K_{\eps_n}} + \frac{1}{K_{\eps_n}} + \mu^-_{\eps_n}((-\infty, 0)) + \mu^+_{\eps_n}((0, \infty))\Big) \\
		&= \lim_{n \to \infty} \frac{\nu_{\eps_n}(\Delta) + 1 + \int_{-\infty}^{0} (1 \wedge -x) \, \nu_{\eps_n}(dx) + \int_{0}^{\infty} (1 \wedge x) \, \nu_{\eps_n}(dx)}{K_{\eps_n}} = \lim_{n \to \infty} \frac{K_{\eps_n}}{K_{\eps_n}} = 1,			
	\end{align*}
	and we finally get
	\begin{align*}
		& c_1 + c^-_2 + c^+_2 + c_3 + \int_{-\infty}^{0} (1 \wedge -x) \, \nu^-(dx) \int_{0}^{\infty} (1 \wedge x) \, \nu^+(dx) \\ &= c_1 + c^-_2 + c^+_2 + c_3 + \mu^-((-\infty, 0)) + \mu^+((0, \infty)) \\
		&= p_1 + p_2 + \mu^-([-\infty, 0]) + \mu^+([0, \infty]) = 1.		
	\end{align*}

	Now, we will show the converse. Let $\mc{D}'$ be the set of functions satisfying Equation \eqref{firstdomain} for $f$ belonging to $C(\R_{\Delta})$ and having second derivatives in $\R \backslash  \{0\}$ which admit an extension in $C(\R_{\Delta})$. We will show that $\mc{D}' \subseteq \mf {D}(\msf{L})$.
	To that end, consider the equation on $\R\backslash \{0\}$, given by
	\begin{equation}
		\label{eqdom}
		\beta f - \frac{1}{2}f'' = g,
	\end{equation}
	for $\beta > 0$ and $g \in C(\R_{\Delta})$. We claim that this equation has a solution in $\mf {D}(\msf{L})$. In fact, by~\cite[Proposition 6.12]{leGall}, using the Feller property established in Proposition~\ref{BMFeller}, we have that $\msf{R}_{\beta}g \in \mf {D}(\msf{L})$ for all $g \in C(\R_{\Delta})$ and the operator $\msf{R}_{\beta}$ is the inverse of $\beta - \msf{L}$. Thus, taking $f = \msf{R}_{\beta}g$, we have
	\begin{equation*}
		\beta f - \frac{1}{2}f'' = (\beta - \msf{L})f = g
	\end{equation*}
	on $\R\backslash \{0\}$, as desired.	
	Now, let $f \in \mc{D}'$, then $\beta f-\frac{1}{2}f''$ is an element of $C(\R_{\Delta})$. Assume that $f$ does not belong to $\mf {D}(\msf{L})$. By the above assertion and the inclusion $\mf {D}(\msf{L}) \subseteq \mc{D}'$ proved before, there are two functions $f_1, f_2 \in \mc{D}'$ satisfying the same equation, that is
	\begin{equation}
		\label{eqdom1}
		\beta (f_1 - f_2) - \frac{1}{2}(f_1'' - f_2'') = 0,
	\end{equation}
	for $x \neq 0$. But we know that the solution for \eqref{eqdom1} is given by
	\begin{equation*}
		f_1(x) - f_2(x) = d_1 \exp(-\sqrt{2\beta}x) + d_2 \exp(\sqrt{2\beta}x),
	\end{equation*}
	for some constants $d_1, d_2 \in \R$. However, $f_1 - f_2$ also goes to zero at infinity. Thus,
	\begin{equation*}
		f_1(x) - f_2(x)	= \begin{cases}
			d_2 \exp(-\sqrt{2\beta}x), \; &\text{ if } x > 0, \\
			d_1 \exp(\sqrt{2\beta}x), \; & \text{ if } x < 0.
		\end{cases}
	\end{equation*}
	
	Moreover, the continuity of $f_1, f_2$ forces $d_1 = d_2 = d$. Also, $d \neq 0$ because $f_1, f_2$ are supposed to be different. Since $f_1, f_2 \in \mc{D}'$, we get
	\begin{align*}
		& c_1\big(f_1(0) - f_2(0)\big) + c^-_2\big(f'_1(0-) - f'_2(0-)\big)\\
		& - c^+_2\big(f'_1(0+) - f'_2(0+)\big) + \frac{c_3}{2}(f''_1(0+) - f''_2(0+)) \\ 
		&= \int_{-\infty}^{0} \Big\{(f_1(x) - f_2(x)) - (f_1(0) - f_2(0)) \Big\}\, \nu^-(dx) \\ 
		&  + \int_{0}^{\infty} \Big\{(f_1(x) - f_2(x)) - (f_1(0) - f_2(0)) \Big\}\, \nu^+(dx).
	\end{align*}
	
	Computing the limits at $0$ for the lateral derivatives of $f_1 - f_2$ and dividing both sides by $d$, we obtain 
	\begin{align*}
		& c_1 + \sqrt{2\beta}(c^-_2 + c^+_2) + c_3\beta \\ &= \int_{-\infty}^{0} \Big\{ \exp(\sqrt{2\beta}x) - 1 \Big\} \, \nu^-(dx) + \int_{0}^{\infty} \Big\{ \exp(-\sqrt{2\beta}x) - 1\Big\} \, \nu^+(dx),
	\end{align*}
	which is a contradiction, since both integrands are non-positive whereas\linebreak $c_1, c^-_2, c^+_2, c_3 \geq 0$, and these constants and the measures $\nu^-(dx), \nu^+(dx)$ cannot be $0$ at the same time, otherwise \eqref{constants normalization} would not be satisfied. Therefore, $f \in \mf {D}(\msf{L})$ and we can conclude.
\end{proof}

We are now ready to prove Theorem~\ref{thm:general_R}. First, we will determine the domain of the infinitesimal generator for a general Brownian motion on $\R_\Delta$. Then, we will show that the operator  defined in Theorem~\ref{thm:general_R} gives us a process which is in fact a general Brownian motion on $\R_\Delta$ with boundary conditions at the origin.

\begin{proof}[Proof of Theorem~\ref{thm:general_R}]
	Let $W$ be a general Brownian motion on $\R_\Delta$ with boundary conditions at the origin. By Lemma \ref{quaseFeller}, it suffices to prove that $\nu^- = \nu^+ = 0$ which we already obtained for $0 \leq \lambda < \infty$. Arguing by contradiction, suppose that $\nu^-(-\infty, -\eps)>0$ and $\nu^+(\eps, \infty) > 0$ for some $\eps > 0$. Then, choosing any $f_1 \in \mf {D}(\msf{L})$, we can slightly modify it to obtain $f_2 = f_1$ on $[-\eps, \eps]$, but $f_2(x) < f_1(x)$ for $|x| > \eps$ so that $f_2$ still belongs to $C(\R_{\Delta})$, $f_2''$ exists in $\R \backslash  \{0\}$ and can be extended to an element of $C(\R_{\Delta})$ (the fact that $f''_2$ has an extension to 0 is obvious since it coincides with $f_1$ near $0$). Thus, according to Lemma \ref{neighb0}, we have $f_2 \in \mf {D}(\msf{L})$. However, since $f_1, f_2$ and its derivatives coincide in $0$, Lemma \ref{quaseFeller} yields
	\begin{align*}
		&\int_{(-\infty, 0)} \big(f_1(x) - f_2(x)\big) \, \nu^-(dx) + \int_{(0,\infty)} \big(f_1(x) - f_2(x)\big) \, \nu^+(dx) \;=\; 0\,.
	\end{align*}
	Hence,
	\begin{align*}
	 \int_{(-\infty, -\eps)} \big(f_1(x) - f_2(x)\big) \, \nu^-(dx) \;=\; - \int_{(\eps,\infty)} \big(f_1(x) - f_2(x)\big) \, \nu^+(dx)\,.
	\end{align*}
Since both integrands are positive, we have that
	\begin{align*}
		&\int_{(-\infty, -\eps)} \big(f_1(x) - f_2(x)\big) \, \nu^-(dx) \;=\; \int_{(\eps,\infty)} \big(f_1(x) - f_2(x)\big) \, \nu^+(dx) \;=\; 0 
	\end{align*}
	implying that $(f_1- f_2)\mathbf{1}_{\{(-\infty, \eps)\}} = 0$ 
	$ \nu^-$-a.e.\  and  $(f_1- f_2)\mathbf{1}_{\{(\eps, \infty)\}} = 0$  $\nu^+$-a.e.,  	which is not possible since both $f_1 - f_2 $ and $\nu^-$ are positive on $(- \infty, \eps)$, and both $f_1 - f_2 $ and $\nu^+$ are positive on $(\eps, \infty)$.
	
	Therefore, $\nu^-(-\infty, -\eps) = \nu^+(\eps, \infty) = 0$ for all $\eps > 0$. Letting $\eps \searrow 0$ we get $\nu^-(-\infty, 0) = \nu^+(0, \infty) = 0$, as desired.
	The expression of the generator was proven in Proposition~\ref{generator}.
For the converse, we will show that the operator $\msf{L}$ on $C(\R_{\Delta})$ with domain $\mf {D}(\msf{L})$ satisfies the hypothesis of the Hille-Yosida Theorem \ref{thm:Hille-Yosida}. This will imply that it will be the generator of a strongly continuous contraction semigroup on $C(\R_{\Delta})$, which is therefore a Feller process. A computation will then show that the semigroup constructed in that way is the semigroup of a general Brownian motion. First, we claim that $\mc{R}(\lambda - \msf{L}) = C(\R_{\Delta})$, for all $\lambda > 0$. Indeed, let $\lambda > 0$ and $g \in C(\R_{\Delta})$. We know that the equation $\lambda f - \frac{1}{2}f'' = g$ has a solution given by
		\begin{equation*}
		f_p(x) = \int \frac{1}{\sqrt{2\lambda}} \exp(-\sqrt{2\lambda}|x - y|) g(y) \, dy,
	\end{equation*}
	which is the resolvent operator of a standard Brownian motion.
	
	Define 
	\begin{equation*}
		f(x) = \begin{cases}
			f_p(x) + A\exp(-\sqrt{2\lambda}x)\,, & \text{ if } x \geq 0,\vspace{3pt} \\
			f_p(x) + A\exp(\sqrt{2\lambda}x) \,, & \text{ if } x < 0.
		\end{cases}
	\end{equation*}
	In that way, $f \in C(\R_{\Delta})$ and $f''$ exists in $\R \backslash  \{0\}$ extending to an element of $C(\R_{\Delta})$. Moreover, for $f$ to satisfy the domain's equation 
	\begin{equation}
		\label{domain's equation R}
		c_1f(0) + c^-_2f'(0-) - c^+_2f'(0+) + \frac{1}{2}c_3f''(0+) = 0
	\end{equation} 
	it suffices to choose $A$ via
	\begin{equation*}
		A = - \frac{c_1f_p(0) + (c^-_2 - c^+_2)f_p'(0) + \frac{1}{2}c_3f_p''(0)}{c_1 + \sqrt{2\lambda}(c^-_2 + c^+_2) + \lambda c_3}.
	\end{equation*}
	
	Thus, $f \in \mf {D}(\msf{L})$. Furthermore, for $x > 0$ or $x < 0$, we have that $\lambda f - \msf{L}f = \lambda f - \frac{1}{2}f'' = g$. Now, for $x = 0$, we have $(\lambda f - \msf{L}f)(0) = \lambda f(0) - \frac{1}{2}f''(0+) = \lambda f_p(0) + \lambda A  - \frac{1}{2}(f_p''(0) + 2\lambda A) = g(0)$. Hence, $g$ belongs to $\mc{R}(\lambda - \msf{L})$. The other inclusion is clear, so the claim is proved. 
	
	Now, we will prove that $\msf{L}$ is dissipative. Let $f \in \mf {D}(\msf{L})$. We may assume that $f \not\equiv 0$. Since $f$ belongs to $C(\R_{\Delta})$, there is $x_1\in\bb R$ such that $|f|$ attains a global maximum in $x_1$. We consider the following cases.
	
	Suppose $x_1 \neq 0$. If $f(x_1) > 0$, then $f(x_1)$ is also a maximum of $f$, and since $f$ is twice differentiable in a neighbourhood of $x_1$, we have that $f''(x_1) \leq 0$. Thus,
		\begin{equation*}
		\Vert \lambda f - \msf{L}f \Vert_\infty \geq \lambda f(x_1) - \frac{1}{2}f''(x_1) \geq \lambda f(x_1) \geq \lambda|f(x)|, \forall \, x \in \R,
	\end{equation*} 
	which implies that $\Vert  \lambda f - \msf{L}f \Vert_\infty \geq \lambda \Vert f \Vert_\infty$. If $f(x_1) < 0$, then $f(x_1)$ is actually a minimum of $f$, and the argument follows similarly noting that $f''(x_1) \geq 0$.
	
For $x_1 = 0$ and $f(0) > 0$, we again have that $0$ is a point of maximum of $f$, and even $f$ might not be differentiable at $0$, we still have that $f'(0-) \geq  0$ and $f'(0+) \leq  0$. Then, the domain's equation \eqref{domain's equation R} forces $f''(0+) \leq 0$, and we can proceed as above. The argument for the case $f(0) < 0$ is analogous, noting that $f(0)$ will be a minimum of $f$.
	
	Finally, we will show that $\mf {D}(\msf{L})$ is dense in $C(\R_{\Delta})$. Let $f \in C(\R_{\Delta})$ and consider the following shift of $f$:
		\begin{equation*}
			f_{\eps} = f(x - \eps)\mathbf{1}_{\{x \geq \eps\}} + f(x + \eps)\mathbf{1}_{\{x \leq -\eps\}} + f(0)\mathbf{1}_{\{|x| < \eps\}}.
		\end{equation*}
Take now $\phi_{\eps}$ a standard approximation of the identity which has a support centred at the origin with radius $\eps$, and consider $f_{\eps} \ast  \phi_{\eps/2}$, which we define to be zero in $\Delta$ and belongs to $C^{\infty}(\R) \cap C(\R_{\Delta})$. By the uniform continuity of $f$, $f_{\eps} \ast  \phi_{\eps/2} \to f$ uniformly as $\eps \rightarrow 0$. Moreover, since $f_\eps$ is constant equal to $f(0)$ in a neighbourhood of zero, we have that $f_{\eps} \ast  \phi_{\eps/2}$ is also constant equals to $f(0)$ around zero. In particular, the derivatives of $f_{\eps} \ast  \phi_{\eps/2}$ vanish at  zero.

Now, we want to correct $f_{\eps} \ast  \phi_{\eps/2}$ around $0$ so that the new corrected function belongs to $\mf {D}(\msf{L})$, while still approximating $f$. To do so, we choose $P$ a polynomial such that 
\begin{equation*}
	P(x) = \begin{cases}
		ax + bx^2\,,& \text{if } x \geq 0, \\
		\tilde{a}x + bx^2\,, &\text{if } x \leq 0,
	\end{cases}
\end{equation*}
for some $\tilde{a}, a, b \in \R$ to be properly chosen. Taking a bump function $\varphi_{\eps} \in C^{\infty}(\R)$ so that $\varphi_{\eps} \equiv 1$ in $(-\eps/2, \eps/2)$ and $\varphi_{\eps}$ vanishes on $\R \backslash (-\eps, \eps)$, we get that
$P\varphi_{\eps} \longrightarrow 0$ uniformly as $\eps \rightarrow 0$, since $P(0) = 0$. Note also that $P\varphi_{\eps} \in C(\R_{\Delta})$ is twice differentiable in $\R \backslash  \{0\}$, and its second derivative admits an extension to an element of $C(\R_{\Delta})$. In particular, the sum $g_{\eps} = f_{\eps} \ast  \phi_{\eps/2} + P\varphi_{\eps}$ keeps this property. Also, a small computation yields that
\begin{equation*}
c_1g_\eps(0) + c^-_2g_\eps'(0-) - c^+_2g_\eps'(0+) + \frac{1}{2}c_3g_\eps''(0+) =c_1f(0) + c^-_2\tilde{a} - c^+_2a + c_3b\,.
\end{equation*}
We can always choose $\tilde{a}, a, b$ so that the right hand side of the above equation is equal to zero, since $c_1, c^-_2, c^+_2, c_3, f(0)$ are fixed, and the case $c^-_2 = c^+_2 = c_3 = 0$ is excluded due to $c_1 \neq 1$. Hence, $g_{\eps} \in \mf {D}(\msf{L})$ and $g_{\eps} \longrightarrow f$ uniformly as $\eps \rightarrow 0$. An application of Proposition~\ref{Equal_BM} completes the proof.
\end{proof}

\section{The most general BM on \texorpdfstring{$\bb G_{\Delta}$}{G Delta}-- Proof of Theorem~\ref{teo:5.5}}\label{s4}
\label{section:BMG}

Since each half-line $(-\infty, 0-]$ and $[0+, \infty)$ can be identified with the standard negative and positive half-lines some results from the previous section  remain in force in this setting. We will state them now.

First, the hitting time of $0$ for Brownian motion is finite almost surely. Moreover, since Brownian motions do not explode in finite time, $W_t$ cannot reach $\Delta$ before~$0-$ or $0+$, which by Definition~\ref{BMonE} are the only points of discontinuity. Also, by Proposition~\ref{prop:jump time} we have that $T_+ = \inf \{t > 0 : W_t \neq 0+\}$ and $T_- = \inf \{t > 0 : W_t \neq 0+\}$ are exponentially distributed under the probabilities $P_{0+}, P_{0-}$, with parameters $\lambda_+$ and $\lambda_-$, respectively. We distinguish the following cases depending on the value of $\lambda_+$. The situation for $\lambda_-$ is analogous.\bigskip

 \textbf{Case 1:} $0 < \lambda_+ < \infty$. 	In this case, by Proposition \ref{prop:jump time}, the process leaves $0+$ via a jump, and the second condition in Definition \ref{BMonE} enables $W$ only to jump to $0-$ or to $\Delta$. Hence, at time $T_+$, the process is killed or arises on the opposite half-line.
	
 \textbf{Case 2:} $\lambda_+ = \infty$.	Here, $P_{0+}(T_+ = 0) = 1$ and the process leaves $0+$ at once. Since the process starts at $0+$ under $P_{0+}$ and the process is càdlàg,  thus right continuous,  the process cannot go immediately from $0+$ to $\Delta$ or $0-$, since this would be a jump and the process would be left-continuous but not right-continuous. 
	
\textbf{Case 3:} $\lambda_+ = 0$.	
	Here, we have $T_+ = \infty$ a.s. and $0+$ is an absorbing point.

Moreover, we can prove that a general Brownian motion on $\bb G_\Delta$ is a Feller process by following the same proof as in Proposition \ref{BMFeller}. One step that is more delicate is to show the continuity of $E_x[f(W_t)]$ in $x \in \bb G_{\Delta}$, for fixed $t \geq 0$ and $f \in C(\bb G_{\Delta})$. Indeed, we can use the same coupling as in \eqref{eq: coupling}, but now we need to be more careful about the discontinuities in the sample paths of $W_t$. To illustrate the argument, consider $W^x_t$ as in \eqref{eq: coupling} for $x \in [0+, \infty)$. In particular, $W^0_t$ refers to a general Brownian on $\bb G_\Delta$ starting from $0+$. Since the sample paths of $W^0_t$ are càdlàg, their discontinuities are at most countable and occur at the stopping times $\tau^n, n \in \N$ given by
\begin{equation*}
\tau^1 = \inf \{s \geq 0 : W^0_s \in \{0-, \Delta\}\}, \tau^2 = \inf \{s \geq \tau^1 : W^0_s \in \{0+, \Delta\}\}, 
\end{equation*}
and so on with the usual convention that the infimum over the empty set is infinity and that $\tau^{n+1}=\infty$ if $\tau^n=\infty$. With this setting and Remark~\ref{continuousrv}, one can apply the same techniques as in the proof of Proposition \ref{BMFeller}.

The reader should also notice that Proposition \ref{generator} was proved considering $x \in [0, \infty)$ or $x \in (-\infty, 0]$ separately and was based mostly on analytic arguments, so it adapts easily to Brownian motion on the state space $\bb G_\Delta$. In fact, in this case, we can say even more.

\begin{proposicao}
	Let $f \in \mf {D}(\msf{L})$, the domain of a general Brownian motion on $\bb G_\Delta$ with boundary conditions at the origin. Then, $f$ is twice differentiable with $f'' \in C(\bb G_\Delta)$. Moreover,
	\begin{equation}
		\label{genBME}
		\msf{L}f(x) = \frac{1}{2}f''(x),
	\end{equation}
for $x \in \bb G$ and $\msf{L}f(\Delta) = 0$.
\end{proposicao}

\begin{proof}
	In view of our previous observations, it remains only to prove that
	\begin{equation*}
		\msf{L}f(0+) = \frac{1}{2}f''(0+), \quad \msf{L}f(0-) = \frac{1}{2}f''(0-)
	\end{equation*}
	for $f \in \mf {D}(\msf{L})$, where the second and first derivatives at $0+$ and $0-$ are interpreted as the right-hand derivative and left-hand derivatives, respectively.
	
	To prove that, we give a completely analytical argument. First, we identify $[0+, \infty]$ with the interval $[0, \infty]$, and apply the mean value theorem to obtain
	\begin{equation*}
		f'_{+}(0) = \lim_{h \searrow 0} \frac{f(h) - f(0)}{h} = \lim_{c \searrow 0} f'(c),
	\end{equation*}
	as was done in Remark~\ref{rem:dif}. 
	Returning to the notation $[0+, \infty]$, we thus have
	\begin{equation*}
		f'(0+) = \lim_{h \searrow 0+} \frac{f(h) - f(0+)}{h} = \lim_{c \searrow 0+} f'(c),
	\end{equation*}
	which shows in particular that $f'$ is continuous in $0+$.
		We now apply the same reasoning again. We have that $f'$ is continuous on $[0, h]$ and differentiable on $(0, h)$, so we can find $c = c(h) \in (0, h)$ such that
	\begin{equation*}
		f''(c) = \frac{f'(h) - f'(0)}{h}.
	\end{equation*}
	We already know that $\lim_{c \searrow 0} f''(c) = 2\msf{L}f(0)$, so we deduce that
	\begin{equation*}
		f''_{+}(0) = \lim_{h \searrow 0} \frac{f'(h) - f'(0)}{h} = 2\msf{L}f(0), 
	\end{equation*}
	i.e.,
	\begin{equation*}
		f''(0+) = \lim_{h \searrow 0+} \frac{f'(h) - f'(0+)}{h} = 2\msf{L}f(0+).
	\end{equation*}
	An identical argument applied to the interval $(-\infty, 0-]$ shows that $f''(0-) = 2\msf{L}f(0-)$.
\end{proof}

The next result states that to show that a function is an element of $\mf {D}(\msf{L})$, it is sufficient to know its behavior in a neighborhood of $0+$ and  $0-$. The proof is almost identical to that of Lemma \ref{neighb0}, and is therefore omitted.

\begin{lema}
	\label{neighb0E}
	If $f_1 \in \mf {D}(\msf{L})$, $f_2, f_2'' \in C(\bb G_\Delta)$, and $f_1 = f_2$ in $(-\eps, 0-] \cup [0+, \eps)$, for some $\eps > 0$, then $f_2 \in \mf {D}(\msf{L})$.
\end{lema}

To obtain the domain's boundary conditions at $0+$ and $0-$, we proceed similarly as we did for the general Brownian Motion on $\R_\Delta$, expressing them first in terms of measures on $(-\infty, 0-]$ and $[0+, \infty)$. The proof is close to the one for general Brownian motion on $\R_\Delta$. However, since it is very technical and adaptations are needed we give a complete proof. 
\begin{lema}
	\label{quaseFellerE}
	There exists non-negative constants $a^+, a^-, c_i^+, c_i^-$, $i = 1, 2, 3$ and measures $\nu_j^+$ on $(0+, \infty)$, $\nu_j^-$ on $(-\infty, 0-)$, $j = 1, 2$ for which
	\begin{align}
		\label{constants E +}
		c_1^+ + a^+ + c_2^+ + c_3^+ + \int_{(0+, \infty)} (1 \wedge x) \, \nu_1^+(dx) + \nu_1^-((-\infty, 0-)) &= 1\,, \\
		c_1^- + a^- + c_2^- + c_3^- + \int_{(-\infty, 0-)} (1 \wedge -x) \, \nu_2^-(dx) + \nu_2^+((0+, \infty)) &= 1\,, \label{constants E -}
	\end{align}
	such that $f$ belongs to $\mf {D}(\msf{L})$ if and only if $f \in C(\bb G_\Delta)$ is twice differentiable with $f''\in C(\bb G_\Delta)$, and it satisfies
	\begin{align}
		\label{firstdomainE}
		& \nonumber c_1^+f(0+) + a^+(f(0+) - f(0-)) - c_2^+f'(0+) + \frac{c_3^+}{2}f''(0+) \\
		&=  \int_{(0+, \infty)} (f(x) - f(0+)) \, \nu_1^+(dx) + \int_{(-\infty, 0-)} (f(x) - f(0+)) \, \nu_1^-(dx), \\ \nonumber
		\\ \label{firstdomainE2}
		& \qquad\text{and}\nonumber \\
		& \nonumber c_1^-f(0-) + a^-(f(0-) - f(0+)) + c_2^-f'(0-) + \frac{c_3^-}{2}f''(0-) \\
		&= \int_{(0+, \infty)} (f(x) - f(0-)) \, \nu_2^+(dx) + \int_{(-\infty, 0-)} (f(x) - f(0-)) \, \nu_2^-(dx).
	\end{align}
\end{lema}

\begin{proof}
	Recall from Proposition~\ref{prop:jump time} that $T_+ = \inf \{t > 0 : W_t \neq 0+\}$ and $T_- = \inf \{t > 0 : W_t \neq 0-\}$ are exponentially distributed with parameters $\lambda_+$ and $\lambda_-$, under the probabilities $P_{0+}, P_{0-}$, respectively. We examine the following cases, supposing first that $f \in \mf {D}(\msf{L})$. \smallskip

	\textbf{Case} $\lambda_+ = 0$ or $\lambda_- = 0$.
	
	If $\lambda_+ = 0$, then $0+$ is an absorbing point, so $f''(0+) = \msf{L}f(0+) = 0$, and we may choose $c^+_1 = a^+ = c^+_2 = 0$,  $c^+_3 = 1$, and $\nu_1^- = \nu_1^+ = 0$. The case $\lambda_- = 0$ is analogous.

	\textbf{Case} $0 < \lambda_+ < \infty$ or $0 < \lambda_- < \infty$:		
	
	Again we will consider only $0 < \lambda_+ < \infty$. By case 1 above, the process waits an exponential time at $0+$, then it jumps to $0-$ or $\Delta$. In particular, $0+$ is not a trap under $P_{0+}$, so we may apply Dynkin's formula~\eqref{teo: Dykin}. 
	
	Let $A = [0+, \eps)$ with $\eps \searrow 0$. Starting from $0+$, the process exits $A$ precisely when it jumps from $0+$ to the points $0-$ or $\Delta$. Thus,
	\begin{align*}
		\msf{L}f(0+) &= \lim_{\eps \searrow 0} \frac{E_{0+}[f(W_{T_+})] - f(0+)}{E_{0+}[T_+]} \\
		&= \lambda_+ [f(0-)P_{0+}(W_{T_+} = 0-) + f(\Delta)P_{0+}(W_{T_+} = \Delta) - f(0+)].
	\end{align*}
	
	Recalling that $f$ vanishes at $\Delta$, we rewrite the equation as follows
	\begin{align*}					
		\msf{L}f(0+)  &= \lambda_+P_{0+}(W_{T_+} = 0-)(f(0-) - f(0+)) - \lambda_+(1-P_{0+}(W_{T_+} = 0-))f(0+) \\
		&= \lambda_+P_{0+}(W_{T_+} = 0-)(f(0-) - f(0+)) - \lambda_+P_{0+}(W_{T_+} = \Delta)f(0+).
	\end{align*}
	
	Since $\msf{L}f(0+) = \frac{1}{2}f''(0+)$, we get
	\begin{equation}
		\label{holdpoint}
		\begin{split}
		\frac{1}{2}f''(0+) &+ \lambda_+P_{0+}(W_{T_+} = \Delta)f(0+)\\ &+ \lambda_+P_{0+}(W_{T_+} = 0-)\big(f(0+) - f(0-)\big) \;=\; 0\,.
		\end{split}
			\end{equation}
	
	Setting $c_1^+ = \frac{\lambda_+P_{0+}(W_{T_+} = \Delta)}{1 + \lambda_+}$, $a^+ = \frac{\lambda_+P_{0+}(W_{T_+} = 0-)}{1 + \lambda_+}$, $c^+_2 = 0$, $c^+_3 = \frac{1}{1 + \lambda_+}$ and $\nu_1^- = \nu_1^+ = 0$, the lemma follows.

	\textbf{Case} $\lambda_+ = \infty$ or $\lambda_- = \infty$.
		If $\lambda_+ = \infty$, then the process leaves $0+$ immediately, so we use Dynkin's formula \eqref{teo: Dykin} again with $A = [0+, \eps)$ with $\eps \searrow 0$. Starting from $0+$, the process can leave $A$ only via $\eps, \Delta, 0-$. Then considering $T_{\eps} = \tau_{\eps} \wedge \tau_{\Delta} \wedge \tau_{0-}$, by Dynkin's Formula, we have
	\begin{equation*}
		\msf{L}f(0+) \;=\; \lim_{\eps \searrow 0} \frac{E_{0+}[f(W_{T_{\eps}})] - f(0+)}{E_{0+}[T_{\eps}]}\,.
	\end{equation*}
	
	We rewrite this equation in terms of the pull-back measure $\nu_{\eps} (dx) = $\break $ \frac{1}{E_{0+}[T_{\eps}]}P_{0+}(W_{T_\eps} \in dx)$ on $\bb G_\Delta$. Thus,
	\begin{equation}
		\label{eqgenE}
		\msf{L}f(0+) \;=\; \lim_{\eps \searrow 0} \, \int_{\bb G} (f(x) - f(0+)) \, \nu_{\eps}(dx) -f(0+)\nu_{\eps}(\Delta)\,.
	\end{equation}
	We consider $K_{\eps} = 1 + \nu_{\eps}(\{\Delta\} \cup (-\infty, 0-]) + \int_{(0+, \infty)} (1 \wedge x) \, \nu_{\eps}(dx)$ and the measures
	\begin{align*}
		\mu^+_{\eps}(A) &= \int_A (1 \wedge x) \, \frac{\nu_{\eps}(dx)}{K_{\eps}} \text{ on } (0+, \infty)\quad \text{ and } \\
		\mu^-_{\eps} (A) &= \frac{\nu_{\eps}(A)}{K_{\eps}} \text{ on } (-\infty, 0-].
	\end{align*}
	
	Identifying $[0+, \infty)$ with $[0, \infty)$ and noticing that $\mu^+_{\eps}((0+, \infty)) \leq 1$, we use Proposition \ref{weakconvergencemeasures} to obtain a sequence $(\eps_n)_n$ and a measure $\mu^+(dx)$ on $[0+, \infty]$ such that
	\begin{equation*}
		\lim_{n \to \infty} \int_{(0+, \infty)} f(x) \, \mu^+_{\eps_n}(dx) \;=\; \int_{[0+, \infty]} f(x) \, \mu^+(dx)\,,
	\end{equation*}
	for all $f \in C[0+, \infty]$. 	Applying the same reasoning and Remark \ref{weakconv2}, we also obtain a measure $\mu^-(dx)$ on $[-\infty, 0-]$ such that, after possibly going over to a subsequence, we have
	\begin{equation*}
		\lim_{n \to \infty} \int_{(-\infty, 0-]} f(x) \, \mu^-_{\eps_n}(dx) \;=\; \int_{[-\infty, 0-]} f(x) \, \mu^-(dx)\,,
	\end{equation*}
	for all $f \in C[-\infty, 0-]$.
	
	Now, we use these limits with appropriate functions. For $f \in \mf {D}(\msf{L})$, the function $\frac{f(x) - f(0+)}{1 \wedge x}$ is continuous on $(0+, \infty)$ and $\lim_{x \to \infty} \frac{f(x) - f(0+)}{1 \wedge x} = -f(0+) < \infty$. Also, $\lim_{x \downarrow 0+} \frac{f(x) - f(0+)}{1 \wedge x} = \lim_{x \downarrow 0+} \frac{f(x) - f(0+)}{x} = f'(0+) < \infty$. Thus, $\frac{f(x) - f(0+)}{1 \wedge x}$ admits an extension in $C[0+, \infty]$, so
	\begin{align*}
		\int_{[0+, \infty]} \frac{f(x) - f(0+)}{1 \wedge x} \, \mu^+(dx) &= \lim_{n \to \infty} \int_{(0+, \infty)} \frac{f(x) - f(0+)}{1 \wedge x} \, \mu^+_{\eps_n}(dx) \\ 
		&= \lim_{n \to \infty} \int_{(0+, \infty)} \frac{f(x) - f(0+)}{1 \wedge x} (1 \wedge x) \, \frac{1}{K_{\eps_n}}\nu_{\eps_n}(dx) \\
		&= \lim_{n \to \infty} \int_{(0+, \infty)} (f(x) - f(0+)) \, \frac{1}{K_{\eps_n}}\nu_{\eps_n}(dx) \\
		&= \lim_{n \to \infty} \int_{[0+, \infty)} (f(x) - f(0+)) \, \frac{1}{K_{\eps_n}}\nu_{\eps_n}(dx)\,,
	\end{align*}		
	where we have used that $1 \wedge x$ is the Radon-Nikodym derivative of $\mu^+_{\eps_n}(dx)$ with respect to $\frac{1}{K_{\eps_n}}\nu_{\eps_n}(dx)$.
	Similarly, one can see that
	\begin{equation*}
		\int_{[-\infty, 0-]} (f(x) - f(0+)) \, \mu^-(dx) =  \lim_{n \to \infty} \int_{(-\infty, 0-]} (f(x) - f(0+)) \, \frac{1}{K_{\eps_n}}\nu_{\eps_n}(dx)
	\end{equation*}
	for all $f \in \mf {D}(\msf{L})$. 	Moreover, the sequences $\frac{\nu_{\eps_n}(\Delta)}{K_{\eps_n}}$ and $\frac{1}{K_{\eps_n}}$ are bounded, so we may assume that there are $p_1, p_2\in \bb R$ such that
	\begin{equation*}
		\lim_{n \to \infty} \frac{\nu_{\eps_n}(\Delta)}{K_{\eps_n}} = p_1 \quad \mbox{ and }\quad  \lim_{n \to \infty} \frac{1}{K_{\eps_n}} = p_2,
	\end{equation*}
	passing to a further subsequence, if necessary. Applying all these limits,  Equation~\eqref{eqgenE} simplifies to
	\begin{align*}
		& p_1f(0+) + p_2\msf{L}f(0+) \\ &=
		\int_{[0+, \infty]} \frac{f(x) - f(0+)}{1 \wedge x} \, \mu^+(dx) + \int_{[-\infty, 0-]} (f(x) - f(0+)) \, \mu^-(dx)\,.
	\end{align*}
	Hence,
	\begin{align*}
		& (p_1 + \mu^+(\infty) + \mu^-(-\infty))f(0+)\\
		& + \mu^-(0-)(f(0+) - f(0-)) -\mu^+(0+)f'(0+) + \frac{p_2}{2}f''(0+) \\ 
		&=	\int_{(0+, \infty)} \frac{f(x) - f(0+)}{1 \wedge x} \, \mu^+(dx) + \int_{(-\infty, 0-)} (f(x) - f(0+)) \, \mu^-(dx).
	\end{align*}
	Defining the constants
	\begin{align*}
		c^+_1 &= p_1 + \mu^+(\infty) + \mu^-(-\infty), \\
		a^+ &= \mu^-(0-), \\
		c^+_2 &= \mu^+(0+), \\
		c^+_3 &= p_2,
	\end{align*}
	and considering the measures
	\begin{align*}
		\nu_1^+(A) &= \int_A \frac{1}{1 \wedge x} \, \mu^+(dx) \text{ on } (0+, \infty), \\
		\nu_1^-(A) &= \mu^-(A) \text{ on } (-\infty, 0-)
	\end{align*}
	yields the desired Equation \eqref{firstdomainE}. It only remains to check \eqref{constants E +}. This follows from
	\begin{align*}
		& p_1 + p_2 + \mu^-([-\infty, 0-]) + \mu^+([0+, \infty]) \\ &= \lim_{n \to \infty} \frac{\nu_{\eps_n}(\Delta)}{K_{\eps_n}} + \frac{1}{K_{\eps_n}} + \mu^-_{\eps_n}((-\infty, 0-]) + \mu^+_{\eps_n}((0+, \infty)) \\
		&= \lim_{n \to \infty} \frac{\nu_{\eps_n}(\Delta) + 1 + \nu_{\eps_n}((-\infty, 0-]) + \int_{(0+, \infty)} (1 \wedge x) \, \nu_{\eps_n}(dx)}{K_{\eps_n}} \;=\; 1\,.	
	\end{align*}
	
	For the case $\lambda_- = \infty$, we proceed similarly. We just emphasizes that, following this procedure, we will obtain an equation taking the form
	\begin{align*}
		& p_1f(0-) + p_2\msf{L}f(0-) \\ &=
		\int_{[-\infty, 0-]} \frac{f(x) - f(0-)}{1 \wedge -x} \, \mu^-(dx) + \int_{[0+, \infty]} (f(x) - f(0-)) \, \mu^+(dx)\,.
	\end{align*}
	Now, since $\lim_{x \uparrow 0-} \frac{f(x) - f(0-)}{1 \wedge -x} = \lim_{x \uparrow 0-} \frac{f(x) - f(0-)}{-x} = - f'(0-)$, this leads to the  change of sign in the coefficient of $c_2^-$, as compared to $c_2^+$ in the statement of the lemma.\bigskip

	For the converse, let $\mc{D}'$ be the set of functions $f, f'' \in C(\bb G_\Delta)$ satisfying Equations \eqref{firstdomainE} and \eqref{firstdomainE2}. We want to show that $\mc{D}' \subseteq \mf {D}(\msf{L})$. In fact, as in the proof of Lemma \ref{quaseFeller}, we have that the equation
	\begin{equation}
		\label{eqdomE}
		\beta f - \frac{1}{2}f'' = g\,,
	\end{equation}
	for $\beta > 0$ and $g \in C(\bb G_\Delta)$, has a solution in $\mf {D}(\msf{L})$.	
	Let $f \in \mc{D}'$. Since $f, f'' \in C(\bb G_\Delta)$, and $\mf {D}(\msf{L}) \subseteq \mc{D}'$, if $f$ does not belong to $\mf {D}(\msf{L})$, then there exist two different functions $f_1, f_2 \in \mc{D}'$ satisfying the same Equation~\eqref{eqdomE}. Consequently,
	\begin{equation}
		\label{eqdom1E}
		\beta (f_1 - f_2) - \frac{1}{2}(f_1'' - f_2'') = 0 \quad \text{ on } \bb G.
	\end{equation}
	
	Now, considering the known solution of \eqref{eqdom1E} and the fact that $f_1 - f_2$ decays to $0$ at infinity, we obtain
	\begin{equation*}
		f_1(x) - f_2(x)	= \begin{cases}
			d^+ \exp(-\sqrt{2\beta}x), \; &\text{ if } x \in [0+, \infty), \\
			d^- \exp(\sqrt{2\beta}x), \; & \text{ if } x \in (-\infty, 0-],
		\end{cases}
	\end{equation*}
	for some $d^+, d^- \in \R$. Here, we used the natural definition $\exp(0+) = \exp(0-) = 1$. Since $f_1$ and $f_2$ are distinct,  $d^+, d^-$ cannot be both equal to $0$.
	
	Since $f_1, f_2 \in \mc{D}'$, subtracting Equation \eqref{firstdomainE} derived from $f_1$ and $f_2$, and computing $f_1 - f_2$ and its derivatives at $0+$ or $0-$ results in
		\begin{align}
		\label{positive}
		& \nonumber c_1^+d^+ + a^+d^+ - a^+d^- + c_2^+d^+\sqrt{2\beta} +c_3^+d^+\beta \\
		&= \int_{(0+, \infty)} \!\!\!d^+\Big\{\exp(-\sqrt{2\beta}x) - 1\Big\} \, \nu_1^+(dx) + \int_{(-\infty, 0-)} \!\!\!\!\! \Big\{d^-\exp(\sqrt{2\beta}x) - d^+\Big\} \, \nu_1^-(dx).
	\end{align}
	
	On the other hand, doing analogous computations using Equation \eqref{firstdomainE2}, we obtain
	\begin{align}
		\label{negative}
		& \nonumber c_1^-d^- + a^-d^- - a^-d^+ + c_2^-d^-\sqrt{2\beta} +c_3^-d^-\beta \\
		&= \int_{(0+, \infty)} \!\!\!\Big\{d^+\exp(-\sqrt{2\beta}x) - d^-\Big\} \, \nu_2^+(dx) + \int_{(-\infty, 0-)} \!\!\!\!\! d^- \Big\{\exp(\sqrt{2\beta}x) - 1\Big\} \, \nu_2^-(dx).
	\end{align}
	
	We will show that these equations cannot be both true. Indeed, suppose that $d^+ = 0$, then $d^- \neq 0$, so \eqref{negative} simplifies to
	\begin{align*}
		& c_1^- + a^- + c_2^-\sqrt{2\beta} +c_3^-\beta \\
		&= \int_{(0+, \infty)} (- 1) \, \nu_2^+(dx) + \int_{(-\infty, 0-)} (\exp(\sqrt{2\beta}x) - 1) \, \nu_2^-(dx),
	\end{align*}
	since both integrands are negative, whereas the constants $c_1^-, a^-, c^-_2, c_3^-$ are all non-negative, the only possibility for the equation to hold is that all these constants and measures $\nu_2^+(dx), \nu_2^-(dx)$ are simultaneously zero. However, \eqref{constants E -} holds, leading to a contradiction.
	
	For $d^- = 0$ the argument is similar, so we can assume $d^+, d^- \neq 0$. Without loss of generality suppose $d^- \leq d^+$, then \eqref{positive} implies
	\begin{align*}
		& c_1^+ + a^+(1 - \frac{d^-}{d^+}) + c_2^+\sqrt{2\beta} +c_3^+\beta \\
		&= \int_{(0+, \infty)} \Big\{\exp(-\sqrt{2\beta}x) - 1\Big\}\, \nu_1^+(dx) + \int_{(-\infty, 0-)} \Big\{\frac{d^-}{d^+} \exp(\sqrt{2\beta}x) - 1\Big\} \, \nu_1^-(dx),
	\end{align*}
	and we have the same problem that the left-hand side of the equation is non-negative, whereas the right side is negative, giving the same contradiction as before. 	
	Therefore, $f \in \mf {D}(\msf{L})$ and the proof is complete.
\end{proof}

We now come to the proof of Theorem~\ref{teo:5.5}. Its proof is similar to the one of Theorem~\ref{thm:general_R}. 

\begin{proof}[Proof of Theorem~\ref{teo:5.5}]
	Let $W$ be a general Brownian motion on $\bb G_\Delta$ with boundary conditions at the origin. We want to show that the measures $\nu_j^+$, $\nu_j^-$, for $j = 1, 2$, in Lemma \ref{quaseFellerE} are identically zero, which was already shown to be true for $\lambda_+, \lambda_- \neq \infty$. To illustrate the argument, let $j = 1$. We assume that 
	$\nu_1^+(\eps, \infty), \nu_1^-(-\infty, -\eps) > 0$, for some $\eps > 0$. Take any $f_1 \in \mf {D}(\msf{L})$ and define a new function $f_2$ such that $f_2 = f_1$ on $[- \eps, 0-] \cup [0+, \eps]$, but $f_2(x) < f_1(x)$ for $|x| > \eps$ in such a way that $f_2$ also belongs to $C(\bb G_\Delta)$ and $f'' \in C(\bb G_\Delta)$. Thus, by Lemma \ref{neighb0E}, we have that $f_2 \in \mf {D}(\msf{L})$ and Lemma \ref{quaseFellerE} applies to both $f_1, f_2$. Since they agree in a neighborhood of $0+$ and $0-$, Equation \eqref{firstdomainE} reduces to		
	\begin{align*}		
		&  \int_{(\eps,\infty)} (f_1(x) - f_2(x))   \, \nu_1^+(dx) = - \int_{(-\infty, -\eps)} (f_1(x) - f_2(x)) \, \nu_1^-(dx).
	\end{align*}
	
	Since both integrands are positive, this leads to
	\begin{align*}
		& \int_{(\eps, \infty)} (f_1(x) - f_2(x)) \, \nu_1^+(dx) = \int_{(-\infty, -\eps)} (f_1(x) - f_2(x)) \, \nu_1^-(dx) = 0 
	\end{align*}
	implying that
	\begin{align*}
		& (f_1- f_2)\mathbf{1}_{\{(\eps, \infty)\}} = 0 \ \nu_1^+\mbox{-a.e.} \text{ and } (f_1- f_2)\mathbf{1}_{\{(-\infty, -\eps)\}} = 0 \ \nu_1^- \mbox{-a.e.},
	\end{align*}
	which is not possible since both the measures and the functions are positive on each of their respective intervals. Hence, $\nu_1^+(\eps, \infty) = \nu_1^-(-\infty, -\eps) = 0$. Since $\eps > 0$ was arbitrary, it follows that $\nu_1^+ \equiv 0$, $\nu_1^- \equiv 0$, as desired.

Conversely, let $\msf{L}$ and $\mf{D}(\msf{L})$ be as in the statement of the theorem. We apply Theorem~\ref{thm:Hille-Yosida} to show that $\msf{L}$ is the generator of a strongly continuous semigroup on $C(\bb G_\Delta)$, which is therefore a Feller process. An application of Proposition~\ref{Equal_BM} will then show that this process corresponds to a general Brownian motion on $\bb G_\Delta$. First, we verify that $\mc{R}(\lambda - \msf{L}) = C(\bb G_{\Delta})$. Let $\lambda > 0$ and $g \in C(\bb G_{\Delta})$. Identifying $0+$ and $0-$ with $0$, consider
\begin{equation*}
	f^+_p(x) = \int_{0}^{\infty} \frac{1}{\sqrt{2\lambda}} \exp(-\sqrt{2\lambda}|x - y|) g(y) \, dy,
\end{equation*}
for $x \in [0+, \infty)$, which is a solution of $\lambda f - \frac{1}{2}f'' = g$ on $[0, \infty)$, and
\begin{equation*}
	f^-_p(x) = \int_{-\infty}^{0} \frac{1}{\sqrt{2\lambda}} \exp(-\sqrt{2\lambda}|x - y|) g(y) \, dy,
\end{equation*}
for $x \in (-\infty, 0-]$, which is also a solution of $\lambda f - \frac{1}{2}f'' = g$ on $(-\infty, 0]$.

Define
\begin{equation*}
	f(x) = \begin{cases}
		f^+_p(x) + A\exp(-\sqrt{2\lambda}x)\,, & \text{ if } x \in [0+, \infty); \\
		f^-_p(x) + B\exp(\sqrt{2\lambda}x) \,, & \text{ if } x \in (-\infty, 0-].
	\end{cases}
\end{equation*}

We see that $f, f'' \in C(\bb G_{\Delta})$ and that $f$ satisfies $(\lambda- \msf{L})f= g$. For $f$ to satisfy the domain's equations
\begin{align}
	\label{domain's equation G}
	c_1^+f(0+) + a^+(f(0+) - f(0-)) - c_2^+f'(0+) + \frac{c_3^+}{2}f''(0+)	&= 0 \quad \text{ and } \\ \nonumber
	c_1^-f(0-) + a^-(f(0-) - f(0+)) + c_2^-f'(0-) + \frac{c_3^-}{2}f''(0-) &= 0,
\end{align}
we choose $A, B$ so that 
\begin{equation*}
	\begin{cases}
		\alpha A + \beta - a^+B = 0 \\
		\gamma B + \delta - a^-A = 0,
	\end{cases}
\end{equation*}
where
\begin{align*}
	\alpha &= c^+_1 + a^+ + \sqrt{2\lambda}c^+_2 + \lambda c^+_3; \\
	\beta &= c_1^+ f^+_p(0) + a^+ (f^+_p(0) - f^-_p(0)) - c_2^+ (f^+_p)'(0) + \frac{c_3^+}{2} (f^+_p)''(0); \\
	\gamma &= c^-_1 + a^- + \sqrt{2\lambda}c^-_2 + \lambda c^-_3; \\
	\delta &= c_1^- f^-_p(0) + a^- (f^-_p(0) - f^+_p(0)) - c_2^- (f^-_p)'(0) + \frac{c_3^-}{2} (f^-_p)''(0).
\end{align*}

Now, to prove that $\msf{L}$ is dissipative, we take $f \in \mf {D}(\msf{L})$, $f \not\equiv 0$, and $x_1 \in \bb G$ such that $|f(x)| \leq |f(x_1)|$, for all $x \in \bb G$. If $x_1 \neq 0+, 0-$, since $f$ is twice differentiable on $\bb G$, the argument follows as in the case of a general Brownian motion on $\R_\Delta$.

Otherwise, suppose that $x_1 = 0+$ and $f(0+) < 0$, then $0+$ is actually a point of minimum of $f$, so even though the derivative $f'(0+)$ is only a right derivative, we still have
$f'(0+) = \lim_{x \searrow 0+} f'(x) \geq 0$, then the first domain's equation in \eqref{domain's equation G} implies that $f''(0+) \geq 0$, and we get
\begin{equation*}
	\| \lambda f - \msf{L}f \| \geq \frac{1}{2}f''(0+) - \lambda f(0+) \geq -\lambda f(0+) \geq \lambda|f(x)|, \forall \, x \in \bb G,
\end{equation*} 
thus $\| \lambda f - \msf{L}f \| \geq \lambda \| f \|$. If $f(0+) > 0$, then $0+$ is a point of maximum of $f$, and we do the same considerations, noting now that $f'(0+) \leq 0$. The argument for $x_1 = 0-$ is similar.

To conclude we will show that $\mf {D}(\msf{L})$ is dense in $C(\bb G_{\Delta})$.  We proceed in a similar way  as we did for $C(\R_{\Delta})$, but taking some precautions since convolutions are well-behaved in $\bb R$ and we are working with two half-lines. For $f \in C(\bb G_{\Delta})$, we define:
\begin{equation*}
	f_{\eps}^+(x) = f(x - \eps)\mathbf{1}_{\{x > \eps\}}  + f(0+)\mathbf{1}_{\{0+ \leq x \leq \eps\}}\,, \quad \text{ for }x\geq 0
\end{equation*}
and $f_{\eps}^+(x) = f_{\eps}^+(-x)$ for $x\leq 0$. In words, the function $f_{\eps}^+(x)$ is obtaining by doing a small shift of $f$ to the right  keeping if constant around zero, and extending it to the full line by reflection. Analogously, we define
\begin{equation*}
	f_{\eps}^-(x) \;=\; f(x + \eps)\mathbf{1}_{\{x < -\eps\}}  + f(0-)\mathbf{1}_{\{-\eps \leq x \leq 0-\}}\quad \text{ for }x\leq 0
\end{equation*}
and $f_{\eps}^-(x) = f_{\eps}^-(-x)$ for $x\geq 0$.
Taking  a standard smooth mollifier $\phi_{\eps}$, define
\begin{equation*}
	f_{\eps} \ast  \phi_{\eps/2} \;=\; \begin{cases}
		f_{\eps}^+ \ast  \phi_{\eps/2}(x), & \text{ if } x \in [0+,\infty),\\
		f_{\eps}^- \ast  \phi_{\eps/2}(x), & \text{ if } x \in (-\infty,0^-].\\
	\end{cases}
\end{equation*}
The function  $f_{\eps} \ast  \phi_{\eps/2}$ is smooth and decays to zero at infinity. Also, $f_{\eps} \ast  \phi_{\eps/2} \to f$ uniformly as $\eps \rightarrow 0$ and $f_{\eps} \ast  \phi_{\eps/2}$ is equal to $f(0+)$ (respectively, equal to $f(0-)$) in a neighborhood of $0+$ (respectively, in a neighborhood of $0-$). In particular, the derivatives of $f_{\eps} \ast  \phi_{\eps/2}$ at $0+$ and $0-$ are null.
Now, we choose $P$ a polynomial such that 
\begin{equation*}
	P(x) \;=\; \begin{cases}
		ax + bx^2\,, \text{if } x \in [0+, \infty); \\
		\tilde{a}x + \tilde{b}x^2, \text{if } x \in (-\infty, 0-],
	\end{cases}
\end{equation*}
for some $\tilde{a}, a, \tilde{b}, b \in \R$ to be properly selected. We also take a bump function $\varphi_{\eps} \in C^{\infty}(\bb G)$ so that $\varphi_{\eps} \equiv 1$ in $(-\eps/2, 0-] \cup [0+, \eps/2)$ and $\varphi_{\eps}$ vanishes on $\R \backslash (-\eps, 0-] \cup [0+, \eps)$. Thus, $P\varphi_{\eps} \in C(\bb G_{\Delta})$ as well as its derivatives and
$P\varphi_{\eps} \longrightarrow 0$ uniformly as $\eps \rightarrow 0$, since $P(0\pm) = 0$.
A small computation yields that
%Writing the domain's equations \eqref{domain's equation G} to $g_{\eps} = f_{\eps} \ast  \phi_{\eps/2} + P\varphi_{\eps}$, it yields
\begin{align*}
	&c_1^+g_\eps(0+) + a^+(g_\eps(0+) - g_\eps(0-)) - c_2^+g_\eps'(0+) + \frac{c_3^+}{2}g_\eps''(0+)	
	\\
	&=c_1^+f(0+) + a^+(f(0+) - f(0-)) - c_2^+a + c_3^+b
\end{align*}
and
\begin{align*}
	&c_1^-g_\eps(0-) + a^-(g_\eps(0-) - g_\eps(0+)) + c_2^-g_\eps'(0-) + \frac{c_3^-}{2}g_\eps''(0-)\\
	&=c_1^-f(0-) + a^-(f(0-) - f(0+)) + c_2^-\tilde{a} + c_3^-\tilde{b}\,.
\end{align*}
We can always choose $\tilde{a}, a, \tilde{b}, b$ so that the right hand side of the above equations are equal to zero, since the cases $c^+_2 = c^+_3 = 0$ and $c_2^- = c_3^- = 0$ are excluded. For that choice, $g_{\eps} \in \mf {D}(\msf{L})$ and $g_{\eps} \longrightarrow f$ uniformly as $\eps \rightarrow 0$, completing the proof. Hence, we can apply the Hille-Yosida Theorem~\ref{thm:Hille-Yosida}. Proposition~\ref{Equal_BM} implies that the process constructed in that way is a general Brownian motion on $\bb G_\Delta$.
\end{proof}

\section{Proofs of functional CLTs}

\begin{proof}[Proof of Theorem~\ref{TCL_R}]
	Let $\msf{L}$ be the generator of the Skew Sticky Killed Brownian Motion, as in the statement of Theorem~\ref{thm:general_R}.
	By 	\cite[Theorem~6.1, page~28 and Theorem~2.11, page~172]{EK}, it is enough to check that there exists a core $\mc C$ for $\msf{L}$ such that, for all $f\in \mc C$, there exists a sequence $(f_n)_{n\in \bb N}$ in $\mf D(\msf{L_n})$ such that 
	\begin{equation}\label{conditions_EK}
		\begin{cases}
			\Vert f_n - \pi_n f\Vert_\infty \to 0\quad \text{ and }\\
			\Vert \msf{L}_n f_n  - \pi_n\msf{L}f\Vert_\infty \to 0\,,
		\end{cases}
	\end{equation}
	where $\pi_n $ is the restriction to $\bb G_{n,\Delta}$, i.e., $\pi_n f = f\big|_{\bb G_{n,\Delta}}$.
	In our scenario, finding the core and these approximating sequences $(f_n)_{n\in \bb N}$ is particularly simple: we can take the core $\mc C$  as the full domain $\mf {D}(\msf{L})$ and the approximating sequence as the restriction of $f$ itself, that is, $f_n = \pi_n f$. This last fact immediately implies that $\Vert f_n - \pi_n f\Vert_\infty = 0$, so we only need to check the second condition in \eqref{conditions_EK}. 
	
	Since  any function $f\in \mf {D}(\msf{L})$ is $C^2$ away from zero, and such that $f$ and $f''$ decay to zero at infinity,  the discrete Laplacian of $f$ approximates the continuous Laplacian of $f$. Since $n^2 \msf{L}_n$ outside $\{\frac{0}{n}, \Delta\}$ is the discrete Laplacian, this observation guarantees that
	\[\sup_{x\in \bb G_{n,\Delta} \backslash \{\frac{0}{n},  \Delta\}} \vert \msf{L}_n f_n (x)  - \pi_n\msf{L}f(x) \vert  \;\longrightarrow\; 0\,,\quad \text{ as } n\to\infty.\]
	Moreover, $f(\Delta) = \msf{L}f(\Delta) = 0$. Therefore, it only remains to understand what  happens at $\pfrac{0}{n}$.  Since $f(\Delta)=0$, by \eqref{RWgm_R} we have that
	\begin{align*}
		&n^2\msf{L}_n f(\pfrac{0}{n})  \;=\; 	-A f(\pfrac{0}{n})+ B_+ n\Big[f\big(\pfrac{1}{n}\big)-f\big(\pfrac{0}{n}\big)\Big] +B_- n\Big[f\big(\pfrac{-1}{n}\big)-f\big(\pfrac{0}{n}\big)\Big]
		 \\
		&\stackrel{n\to\infty}{\longrightarrow}  	-Af(0)+ B_+ f'(0+) 	- B_- f'(0-) \;=\; \frac{1}{2} f''(0)
	\end{align*}
	due to \eqref{eq: gerador_R} and the definition of $\mf {D}(\msf{L})$ stated in Theorem~\ref{thm:general_R}. This concludes the proof.
\end{proof}

\begin{proof}[Proof of Theorem~\ref{TCL}] The proof is similar to the proof of Theorem~\ref{TCL_R}.
	Let $\msf{L}$ be the generator of the Skew Sticky Killed Snapping Out Brownian Motion, as in the statement of Theorem~\ref{teo:5.5}. It is enough to check that there exists a core $\mc C$ for $\msf{L}$ such that, for all $f\in \mc C$, there exists a sequence $(f_n)_{n\in \bb N}$ in $\mf D(\msf{L_n})$ such that 
	\begin{equation*}
		\begin{cases}
			\Vert f_n - \pi_n f\Vert_\infty \to 0\quad \text{ and }\\
			\Vert \msf{L}_n f_n  - \pi_n\msf{L}f\Vert_\infty \to 0\,,
		\end{cases}
	\end{equation*}
	where $\pi_n $ is the restriction to $\bb G_{n,\Delta}$ that is, $\pi_n f = f\big|_{\bb G_{n,\Delta}}$. Take the core $\mc C$  as the full domain $\mf {D}(\msf{L})$ and the approximating sequence as the restriction of $f$ itself. Again,  since $n^2 \msf{L}_n$ outside $\{\frac{0+}{n}, \frac{0-}{n}, \Delta\}$ is the discrete Laplacian, 
	\[\sup_{x\in \bb G_{n,\Delta} \backslash \{\frac{0+}{n}, \frac{0-}{n}, \Delta\}} \vert \msf{L}_n f_n (x)  - \pi_n\msf{L}f(x) \vert  \;\longrightarrow\; 0\,,\quad \text{ as } n\to\infty\]
	and $f(\Delta) = \msf{L}f(\Delta) = 0$, so we only need to analyse what  happens at $\pfrac{0+}{n}$ and $\pfrac{0-}{n}$. We will study only $\pfrac{0+}{n}$, the case $\pfrac{0-}{n}$ is totally analogous. Since $f(\Delta)=0$, by \eqref{RWgm} we have that
	\begin{align*}
		&n^2\msf{L}_n f(\pfrac{0+}{n})  \;=\; 	-A_+f(\pfrac{0+}{n})+ B_+ n\Big[f\big(\pfrac{1}{n}\big)-f\big(\pfrac{0+}{n}\big)\Big]
		+ C_+\Big[f\big(\pfrac{0-}{n}\big)-f\big(\pfrac{0+}{n}\big)\Big]\\
		&\stackrel{n\to\infty}{\longrightarrow}  	-A_+f(0+)+ B_+ f'(0+) 	- C_+\big[f(0+) - f(0-)\big] \;=\; \frac{1}{2} f''(0+)
	\end{align*}
	due to \eqref{eq: rel_1} and the definition of $\mf {D}(\msf{L})$ stated in Theorem~\ref{teo:5.5}. This concludes the proof.
\end{proof}

\appendix

\section{Some tools on semigroups and generators}\label{app1}

\begin{teorema}[Hille-Yosida Theorem, see Theorem 1.2.6 of \cite{EK} for instance]
	\label{thm:Hille-Yosida}
	A linear operator $L$  is the generator of a strongly continuous contraction semigroup on a Banach space ${B}$ with norm $\Vert \cdot \Vert$ if and only if
	\begin{itemize}
		\item $\mc{D}(L)$ is dense in ${B}$.
		\item $L$ is dissipative, that is $\| \lambda f - L f \| \geq \lambda \| f \|$, for every $\lambda > 0$ and $f \in \mc{D}(L)$.
		\item $\mc{R}(\lambda - {L})$ the range of $\lambda - {L}$ is equal to ${B}$, for some $\lambda > 0$.
	\end{itemize} 
\end{teorema}

 We believe that next result is known, but since we could not find a proof in the literature, we provide one below. In plain words, it says that if a process has the same generator as a Brownian motion in an interval $[a,b]$, then it behaves as a Brownian motion in the interval $[a,b]$.

\begin{proposicao}\label{Equal_BM}
	Let $\{X_t^x:t\geq 0\}$ be a Feller process on an interval $I\subset \bb R$ with generator $\msf{L}$ and domain $\mf D(\msf{L})$, where $x\in I$ denotes its starting point. Fix an interval $[a,b]\subset I$. Assume that 
	\begin{enumerate}[(a)]
		\item For any $\eps>0$, the restriction of functions $f\in  \mf D(\msf{L})$ to the interval $[a+\eps,b-\eps]$ contains the set of functions $g\in C^2[a+\eps,b-\eps]$ such that $g''(a+\eps)=g''(b-\eps)=0$.
		\item For any $f\in \mf D(\msf{L})$ and any $x\in (a,b)$, it holds that $\msf{L} f(x) = \frac{1}{2}f''(x)$. 
	\end{enumerate}
	
	Then, for any $x\in [a,b]$, it holds that  \[
	\{X_{t\wedge \tau_{a,b}}^x:t\geq 0\} \overset{d}{=} \{B_{t\wedge \tau_{a,b}}^x:t\geq 0\}\,,\]
 where $B_t^x$ is a standard Brownian motion starting from $x\in [a,b]$ and $\tau_{a,b}$ is the hitting time of $\{a,b\}$. 
\end{proposicao}

\begin{proof}
	Dynkin's martingale associated to $X_t^x$ is 
	\begin{equation*}
		M_t = f(X_t^x) - f(X_0^x) - \int_0^t \msf{L}f(X_s)ds\,. 
	\end{equation*} 
	Since $f$ and $\msf L f$ are bounded, this martingale is uniformly integrable, so the stopped process
		\begin{equation*}
		\widehat{M}_t = f(X_{t\wedge \tau_{a+\eps, b-\eps}}^x) - f(X_0^x) - \int_0^{t} \msf{L}f(X_{s\wedge \tau_{a+\eps, b-\eps}}^x)ds
	\end{equation*}
	is also a martingale, where $\tau_{a+\eps, b-\eps}$ is the hitting time of $\{a+\eps, b-\eps\}$.  By the Martingale Problem (see \cite[Chapter 4]{EK}) and (a) and (b) above, we conclude that  $X_{t\wedge \tau_{a+\eps,b-\eps}}$ has the same generator as $\{B_{t\wedge \tau_{a+\eps,b-\eps}}^x:t\geq 0\}$, whose  generator  consists of the $C^2$-functions whose second derivatives at the boundary are zero, see \cite[Example 2.10]{Liggett} for instance. Hence
	$\{X_{t\wedge \tau_{a+\eps,b-\eps}}^x:t\geq 0\} \overset{d}{=} \{B_{t\wedge \tau_{a+\eps,b-\eps}}^x:t\geq 0\}$
	for any $\eps>0$. A coupling allows to extend the equality above to $\eps=0$.
\end{proof}

\section*{Acknowledgments}
D.E.\ and T.F. were supported by the National Council for Scientific and Technological Development - CNPq via a Bolsas de Produtividade  303348/2022-4 and 306554/2024-0, respectively. D.E.\ moreover acknowledge support by the Serrapilheira Institute (Grant Number Serra-R-2011-37582). D.E and T.F.\  acknowledge support by the National Council for Scientific and Technological Development - CNPq via two Universal Grants (Grant Number 406001/2021-9 and 401314/2025-1). D.E. and T.F\ were
partially supported by FAPESB (Edital FAPESB Nº 012/2022 - Universal - NºAPP0044/2023). W.M.\ thanks CAPES (Coordena\c c\~ao de Aperfeiçoamento de Pessoal de Nível Superior) for the support via a master scholarship and a PhD scholarship.
The authors are grateful to Milton Jara, Otávio Menezes and Pedro Cardoso for providing useful comments on a previous version of the article.

\bibliography{bibliografia}
\bibliographystyle{plain}
\end{document}